\documentclass[11pt,letterpaper]{amsart}
\usepackage{amssymb}
\usepackage{amsfonts}
\usepackage{amsmath}
\usepackage{graphicx}
\usepackage{xcolor}
\usepackage{amsmath}
\usepackage{mathrsfs}
\usepackage{stmaryrd}
\usepackage{epsfig,color}
\usepackage{blindtext}
\usepackage{enumerate}
\usepackage[citecolor=blue,colorlinks]{hyperref}
\usepackage{tikz-cd}

\usepackage{url}
\usepackage{bbm}
\usepackage{nicefrac,mathtools}
\usepackage{bm}   
\DeclareGraphicsExtensions{.pdf,.jpeg,.png}
\usepackage{epstopdf}
\usepackage{cancel} 
\usepackage[normalem]{ulem} 
\usepackage{verbatim} 
\usepackage{enumitem} 
\usepackage{soul}
\usepackage{color}

\usepackage{footnote}
\usepackage[msc-links, lite]{amsrefs}

\usepackage{geometry}
\newtheorem{theorem}{Theorem}[section]

\newtheorem{proposition}[theorem]{Proposition}
\newtheorem{lemma}[theorem]{Lemma}
\newtheorem{corollary}[theorem]{Corollary}
\newtheorem{question}[theorem]{Question}
\newtheorem{example}[theorem]{Example}

\theoremstyle{definition}
\newtheorem{definition}[theorem]{Definition}
\newtheorem{remark}[theorem]{Remark}
\numberwithin{equation}{section}

\newcommand{\mb}{\mathbb}
\newcommand{\mc}{\mathcal}

\newcommand{\eps}{\varepsilon}

\newcommand{\Sc}{\mathrm{Sc}}
\newcommand{\mr}{\mathrm}

\DeclareMathOperator{\Ric}{Ric}

\newcommand{\vol}{\mathrm{vol}}

\newcommand{\N}{\mathbb{N}}

\newcommand{\R}{\mathbb{R}}
\renewcommand{\subset}{\subseteq}
\newcommand{\defeq}{\mathrel{\mathop:}=}

\newcommand{\haus}{\mathcal{H}}
\newcommand{\leb}{\mathcal{L}}

\newcommand{\dist}{\mathsf{d}}

\newcommand{\meas}{\mathfrak{m}}

\DeclareMathOperator{\RCD}{RCD}
\DeclareMathOperator{\ncRCD}{ncRCD}

\def\Xint#1{\mathchoice
{\XXint\displaystyle\textstyle{#1}}%
{\XXint\textstyle\scriptstyle{#1}}%
{\XXint\scriptstyle\scriptscriptstyle{#1}}%
{\XXint\scriptscriptstyle\scriptscriptstyle{#1}}%
\!\int}
\def\XXint#1#2#3{{\setbox0=\hbox{$#1{#2#3}{\int}$ }
\vcenter{\hbox{$#2#3$ }}\kern-.6\wd0}}

\def\dashint{\Xint-}

\title{ Two-dimension vanishing, splitting and positive scalar curvature}
\author{Xingyu Zhu}\address{Institut f\"ur Agewandte Mathematik, Universit\"at Bonn, Bonn, Germany}\email{zhu@iam.uni-bonn.de}


\begin{document}
\maketitle

\begin{abstract}
We prove several analogs of Gromov's macroscopic dimension conjecture with extra curvature assumptions. More explicitly, we show that for an open Riemannian $n$-manifold $(M,g)$ of nonnegative Ricci (resp. sectional) curvature, if it has uniformly positive scalar curvature and it is uniformly volume noncollapsed, then the essential (resp. Hausdorff) dimension of an asymptotic cone, as a notion of largeness, has a sharp upper bound $n-2$, which is $2$ less than the upper bound for an open Riemannian manifold with only nonnegative Ricci curvature. As a consequence, the dimension of space of linear growth harmonic functions of $M$ has upper bound $n-1$ which is also $2$ less than the sharp bound $n+1$ when $M$ only has nonnegative Ricci curvature. We also prove the first Betti number upper bound is $n-2$ if $M$ is compact, and $n-3$ if $M$ is non-compact. When $M$ is compact we show a fibration theorem over torus, and a rigidity theorem for the fiber when the first Betti number upper bound is achieved.

\end{abstract}

\section{Introduction}\label{sec:intro}
Gromov's macroscopic dimension conjecture roughly states that an open manifold of dimension $n$, $n\ge 3$, with uniformly positive scalar curvature has ``dimension at large'' at most $n-2$. Here the largeness can be interpreted precisely in several equivalent ways, see \cites{Cai_Largemanifold,Shen_Largemanifold}. We refer the readers to \cite{Gromov_four_lectures} for a comprehensive exposition of the geometry of scalar curvature.

The principle or heuristic behind Gormov's conjecture can be described as that uniformly positive scalar curvature implies the vanishing of $2$ dimensions in large scales, where the dimension is a vague notion to be interpreted properly. Instead of the macroscopic dimension, some other types of quantities serving as notions of largeness can be considered. Then we can ask 
\begin{question}\label{ques:main}
    What are the quantities as notions of largeness whose upper bound become $2$ less when imposing a uniform positive scalar curvature lower bound?
\end{question}
 We consider this question under a priori Ricci or sectional curvature lower bound $0$. In such a situation, polynomial growth order of the volume of geodesic balls is a candidate, which is a conjecture of Gromov and is considered in \cites{Zhu_Geometryofpsc, MW_geometryofpsc, Chodosh_Li_Soapbubble} when the dimension is $3$. We follow this line of thoughts and include also Ricci limit spaces into our considerations, since we have nice stability of lower Ricci or sectional curvature bounds under pointed (measured) Gromov-Hausdorff convergence. Inspired and guided by Gromov's principle, in a recent preprint \cite{WZZZ_PSC_RLS}, Wang-Xie-Zhu and the author proved a theorem for Ricci limit spaces as another possible answer to the above Question \ref{ques:main}, where the notion of largeness is chosen to be the number of lines split from a Ricci limit space. In particular, with the crucial use of the torical band estimates \cite{Gromov_metric_inequality}, we showed 

\begin{theorem}[\cite{WZZZ_PSC_RLS}*{Theorem 1.1}]\label{thm:n-2splitting}
    Let $(M_i,g_i,p_i)$ be a sequence of pointed open Riemannian $n$-manifolds of nonnegative Ricci curvature. If $\Sc_{g_i}\ge K>0$ and $\vol_{g_i}(B_1(p_i))>v>0$, then any pointed Gromov-Hausdorff limit of this sequence can split at most $\R^{n-2}$. Moreover, if the maximal splitting happens, the non-splitting factor is either $\mb S^2$ or $\R P^2$ with a metric of nonnegative curvature in the sense of Alexandrov, which is compact. 
\end{theorem}

 As a by product of Theorem \ref{thm:n-2splitting} above, it is also shown in \cite{WZZZ_PSC_RLS}*{Theorem 1.9} that if $(M,g)$ is an open $n$-manifold of nonnegative Ricci curvature and $\Sc_g \ge K>0$, then for any $p\in M$, its asymptotic volume ratio is zero, i.e.,
\begin{equation}\label{eq:AVR0}
  \lim_{r\to\infty}  \frac{\vol_g(B_r(p))}{r^n}=0.
\end{equation}
 This means that any asymptotic cone must collapse in the presence of a positive scalar curvature. It is then natural to wonder if it must collapse at least $2$ dimensions, since the dimension of an asymptotic cone also describes the largeness of a manifold. Having this idea in mind, we proceed to the main results. 
 
 In this note, our focus is to provide some quantities as candidates of answers to Question \ref{ques:main} and also as applications of Theorem \ref{thm:n-2splitting}. More explicitly, we prove that some dimension upper bound for an open manifold of nonnegative Ricci curvature becomes $2$ less if this manifold also has positive scalar curvature. Here, the dimension is interpreted as the dimension of the space of linear growth harmonic functions in Theorem \ref{thm:dim_n-1}, as the essential or Hausdorff dimension of an asymptotic cone as a Ricci limit space or an Alexandrov space in Theorem \ref{thm:codim2collapseRic} and Theorem \ref{thm:codim2collapse}. The new challenge is that we need to tackle the collapsed asymptotic cones where Theorem \ref{thm:n-2splitting} is not directly available. 
 
 To state our first main theorem we fix some notations. For a Riemannian manifold $(M,g)$ let $h_1(M)$ be the dimension of the space of linear growth harmonic functions and $b_1(M)$ be the first Betti number of $M$. 

 \begin{theorem}\label{thm:dim_n-1}
     Let $(M,g)$ be an open Riemannian $n$-manifold of nonnegative Ricci curvature. If $\Sc_g \ge K>0$ and $\vol_{g}(B_1(x))\ge v>0$ for every $x\in M$, then $h_1+b_1\le n-1$ and $b_1(M)\le n-3$. 
 \end{theorem}
 
 It is well-known (\cites{LiTamLinear,hua2016harmonic}) that for an open Riemannian $n$-manifold of nonnegative Ricci curvature, the dimension of linear growth harmonic functions is bounded above by $n+1$, which is closely tied to the splitting in asymptotic cones, discovered by Cheeger--Colding--Minicozzi \cite{CheegerColdingMinicozzi}, see section \ref{subsec:AsymCone}. We see that $2$ dimensions vanish because of the uniformly positive scalar curvature. The equality can be achieved by $\R^{n-2}\times\mb S^2$. 

 \begin{remark}
     Theorem \ref{thm:dim_n-1} is strengthened thanks to Pan--Ye that appears when this note is under review.
 \end{remark}
 
 \begin{question}
 Is Theorem \ref{thm:dim_n-1} true without the uniformly volume noncollapsed assumption,\ i.e. $\vol_g(B_1(x))>v>0$ for any $x\in M$? Is $\R^{n-2}\times\mb S^2$ the only space, up to diffeomorphism, that achieves the upper bound $n-1$? 
 \end{question}
 
  Theorem \ref{thm:dim_n-1} will actually be a corollary of the next theorem, which concerns the collapsing of asymptotic cones.

\begin{theorem}\label{thm:codim2collapseRic}
    Let $(M,g)$ be an open Riemannian $n$-manifold of nonnegative Ricci curvature. If $\Sc_{g}\ge K>0$ and $\vol_{g}(B_1(x))\ge v>0$ for every $x\in M$, then the essential dimension of any asymptotic cone of $M$ is at most $n-2$. 
\end{theorem}

Compared to manifolds with only nonnegative Ricci curvature, whose asymptotic cones can have essential dimension $n$ if at some point the asymptotic volume ratio is positive, $2$ dimensions vanish because of the uniformly positive scalar curvature.

 The next main theorem is to claim that the first betti number is also a candidate of the answers to Question \ref{ques:main}. We present here an upper bound of the first Betti number with positive scalar curvature and then discuss its rigidity. Recall that we denote by $b_1(M)$ the first Betti number of $M$. Let us relax the curvature bound in \cite{WZZZ_PSC_RLS}*{Corollary 1.4} and prove a rigidity theorem. 

\begin{theorem}\label{thm:betti}
     Let $(M,g)$ be a complete compact Riemannian $n$-manifold, with $\Sc_g\ge K>0$, $\mr{diam}(M,\dist_{g})< D<\infty$, and $\vol_g(M)>v>0$. There exists $\eps\defeq \eps(n,D,v,K)>0$, so that if $\Ric_{g}\ge -\eps$, then $b_1(M)\le n-2$. 
     
     If $b_1(M)=n-2$, and in addition 
     $\Ric_g$ is bounded from above, then $M$ is homeomorphic to a fiber bundle over $\mb T^{n-2}$, with the fiber $F$ being $\mb S^2$ or $\R P^2$. 
     
     If $M$ is orientable and a $\mb S^2$ bundle over $\mb T^{n-2}$, where $3\le n\le 7$, then we have that the minimal area of any isometrically embedded $\mb S^2$ is bounded above by $\frac{8\pi}{K}$. 
\end{theorem}

Theorem \ref{thm:betti} is a compact counterpart of Theorem \ref{thm:n-2splitting}. In particular, the following corollary will be clear from the proof of Theorem \ref{thm:fibration} and Theorem \ref{thm:betti}.

\begin{corollary}
    Let $(M_i,g_i)$ be a sequence of complete compact Riemannian $n$-manifolds, with $0\le\Ric_{g_i}\le K_1$, $\Sc_{g_i}\ge K_2>0$, $\mr{diam}(M,\dist_{g_i})< D<\infty$, and $\vol_{g_i}(M)>v>0$. If $(M_i,g_i)$ GH converges to a noncollapsed limit space $(X,\dist)$, then $b_1(X)\le n-2$, if $b_1(X)=n-2$, then $X$ is a topological fiber bundle over $\mb T^{n-2}$, and the fiber is of topological type $\mb S^2$ or $\R P^2$.
\end{corollary}

Although Theorem \ref{thm:codim2collapseRic} implies Theorem \ref{thm:codim2collapse} below, we state it separately and give a proof with only Alexandrov geometry tools. Because the argument is more elementary and encodes a nice relation between the asymptotic cones and the limit space at infinity.

\begin{theorem}\label{thm:codim2collapse} 
Let $(M,g)$ be an open Riemannian $n$-manifold of nonnegative sectional curvature. If $\Sc_{g}\ge K>0$ and $\vol_{g}(B_1(x))\ge v>0$ for every $x\in M$. Then the Hausdorff dimension of the asymptotic cone of $M$ at any point is at most $n-2$.
\end{theorem}



 In the sequel, for a metric space $(X,\dist)$ we say a sequence of points $\{p_i\}_{i\in N}$ in $X$ diverges to infinity, denoted by $p_i\to \infty$, if $\dist(x,p_i)\to \infty$ for any fixed $x\in X$ as $i\to \infty$. We call a pGH limit of $(M,g, p_i)$ a limit space at infinity of $M$ if $p_i\to \infty$.

 All of the proofs originate from the very same general idea, which is to find a large enough number of splitting factors in the limit space at infinity from the information of a collapsed asymptotic cone if the statements to be proved were not true, resulting in a contradiction to \ref{thm:n-2splitting}. 
 
 The relation between an asymptotic cone and a limit space at infinity is straightforward in the presence of nonnegative sectional curvature, as can be seen from Theorem \ref{thm:splitting}. The very reason is the monotonicity of comparison angles, which helps passing information of geodesics on a large scale to smaller ones. This is the content of section \ref{sec:line}.

However, for nonnegative Ricci curvature, the asymptotic cones at a point can be non-unique, even non-homeomorphic \cite{CN11}. For collapsed asymptotic cones, they are not necessarily metric cones, and they can even be nonpolar \cite{MenguyNonpolar}. Moreover, the singular set can have larger Hausdorff dimension than that of the regular set \cite{PanWeiHaus}. As opposed to nonnegative sectional curvature case, the line splitting argument as in section \ref{sec:line} seems to be hard to proceed, see Remark \ref{rem:ric}.

For our purpose, we only need to deal with the simplest possible codimension $1$ collapsing. In this case, a theorem of Kapovitch-Wilking \cite{KapovitchWilking}*{Theorem 5} can be applied to obtain a noncollapsed limit space from a collapsed asymptotic cone by a rescaling, but there is no a priori information for the rescaling sequence. The new observation is that this rescaling argument can be done so that the uniformly positive scalar curvature is preserved.
We also use this idea to study the first Betti number for compact manifolds. These all have the $2$-dimension vanishing behavior hence are put together in section \ref{sec:collapse}. 

In order to to study the rigidity of the first Betti number we also establish a fibration theorem for large first Betti number that maybe interesting in its own right. The fibration theorem is based on the observation that if the limit space is coming from a sequence of manifolds with Ricci curvature bounded above as well as almost nonnegative and the first Betti number almost reach its upper bound, then it is homeomorphic to an manifold. In fact, with the help of splitting theorem we can show that if the first Betti number is large then the universal cover of the limit space is a manifold and the limit space itself is an orbit space of a nice isometric group action on its universal cover, hence also a manifold. This is built on the identification between the deck transform group and the fundamental group for Ricci limit spaces \cite{Wang_RicciSemilocal}, the equivariant convergence of the universal cover \cite{PanWang_universal} and codimension 4 theorem \cite{CN15_codim4}. We separate it in Section \ref{sec:Betti}.

\textbf{Acknowledgements.} The author would like to thank Bo Zhu for fruitful discussions, from which stems the main topics of this work. The author is also grateful to Shouhei Honda for helpful suggestions on an early draft of this work, to Jikang Wang for explaining \cites{PanWang_universal,Wang_RicciSemilocal} as well as for several insights about the proof of Theorem \ref{thm:fibration}, to Igor Belegradek for communicating to the author a part of proof in Theorem \ref{thm:betti}, to Zetian Yan for useful disscusions over \cite{gromov2020NoPcs5D} and \cite{ZhuJT_area}.

\section{Preliminaries and notations}\label{sec:prelim}

\subsection{Structure of metric measure spaces with lower Ricci curvature bounds}
When we speak of an $\RCD(K,N)$ space, we always assume $K\in \R$ and $N\in [1,\infty)$. When speaking of an Aleksandrov space, we always assume it is complete and finite dimensional. It is shown in \cite{PetruninCD} that an Alexandrov space along with its Hausdorff measure is an $\RCD$ space. So what we say about $\RCD$ spaces is also true for Alexandrov spaces. We recall here some basic definitions and facts.

 We start with $\RCD$ spaces. Given an $\RCD(K,N)$ space $(X,\dist,\meas)$, let $\mathcal{R}_k(X)$ be the set of points at which the tangent cone is $(\R^k,|\cdot|,\leb^k)$, for $k\in [1,N]\cap \N$, called the $k$-regular set of $X$, and $\mathcal{R}(X)\defeq \cup_k \mathcal{R}_k$ is called the regular set of $X$.
 
 \begin{definition}[\cite{BrueSemola20Constancy}]
 there is a unique $n\in [1,N]\cap \N$ such that $\meas(X\setminus\mathcal{R}_n)=0$. This $n$ is called the essential dimension of $(X,\dist,\meas)$, denoted by $\mr{essdim}(X)$. 
 \end{definition}
 
 The essential dimension can also be defined as the maximum integer $k$ so that $\R^k$ can be split off in a tangent cone, as a result of \cite{kitabeppu2017sufficient}. Meanwhile, the essential dimension is lower semicontinuous under pGH convergence.
 
\begin{theorem}[\cite{kitabeppu2017sufficient}]\label{thm:lsc}
    Let $(X_i,\dist_i,\meas_i,x_i)$ be a convergent pGH sequence of $\RCD(K,N)$ spaces with limit $(X,\dist,\meas,x)$. Then
    \[
    \mr{essdim}(X)\le \liminf_{i\to\infty} \mr{essdim}(X_i).
    \]
\end{theorem}
 
 An $\RCD(K,N)$ space $(X,\dist,\meas)$ is called noncollapsed, denoted by $\ncRCD(K,N)$, if $\meas=\haus^N$, otherwise it is called collapsed. It follows that $N\in \N$, and $\mr{essdim}(X)=N$, see \cite{DPG17}. A typical example of a $\ncRCD(0,n)$ space is an Riemannian $n$-manifold $(M,\dist_g,\vol_g)$ of nonnegative Ricci curvature. 
 For a noncollapsed space, the regular points can be characterized by the volume density \cite{DPG17}*{Corollary 1.7}:
\begin{equation}\label{eq:density}
    \Theta_N(x)=1 \Leftrightarrow x\in \mathcal{R}_N=\mathcal{R}.
\end{equation}

 Also for a noncollapsed space, the singular set $\mc{S}\defeq X\setminus \mc{R}(X)$ is stratified into 
\[
\mathcal{S}_0\subset \mathcal{S}_1\subset \cdots\subset \mathcal{S}_{N-1},
\]
where for $0\le k\le N-1$, $k\in \mathbb{Z}$, $\mathcal{S}_k=\{x\in \mathcal{S}: \text{no tangent cone at $x$ is isometric to } \R^{k+1}\times C(Z)\text{ for any metric space } Z\}$, where $C(Z)$ is the metric measure cone over a metric space $Z$.
 
The main theorem in \cite{BrenaGigliHondaZhu} can be rephrased as a essential dimension gap theorem, which is proved in \cite{CN12} for Ricci limit spaces.
 
 \begin{theorem}[\cite{BrenaGigliHondaZhu}*{Theorem 1.5}]\label{thm:essdimgap}
     If an $\RCD(K,N)$ space $(X,\dist,\meas)$ is collapsed then its essential dimension is at most $[N]-1$.
 \end{theorem}

 It follows from the stability of $\RCD(K,N)$ condition, see for example \cite{GMS13}, that Ricci limit spaces are $\RCD$ spaces. Furthermore, noncollapsed Ricci limit spaces are $\ncRCD$ spaces, because in \cite{DPG17}, the notion of noncollapsed Ricci limit spaces defined by Cheeger-Colding \cite{Cheeger-Colding97I} and notion of noncollapsed $\RCD$ spaces are unified in the following sense.

 \begin{theorem}\label{thm:ncRCD}
     Let $(X_i,\dist_i,\haus^N,x_i)$ be a sequence of pointed $\ncRCD(K,N)$ spaces. If $(X_i,\dist_i,x_i)$ pGH converges to $(X,\dist, x)$. Then exactly one the following happens.
     \begin{itemize}
         \item $\liminf_{i\to\infty}\haus^N(B_1(p))>0$, and $(X,\dist,\haus^N, x)$ is $\ncRCD(K,N)$. $(X_i,\dist_i,\haus^N, x_i)$ pmGH converges to $(X,\dist,\haus^N, x)$. In particular, the volume converges, i.e.,
         \[
        \lim_{i\to\infty}\haus^N(B_R(x_i))=\haus^N(B_R(x)), \quad \forall R>0.
         \]
         \item  $\liminf_{i\to\infty}\haus^N(B_1(p))=0$, and the Hausdorff dimension of $(X_i,\dist_i)$ is at most $N-1$.
     \end{itemize}
 \end{theorem}

 \subsection{Monotonicity of Angles} 
See for example \cite{BBI01} or \cite{antonelli2023isoperimetric}*{section 2.1}. Let $(A,\dist_A)$ be an Alexandrov space of nonnegative curvature. Recall that given $2$ unit speed rays $\sigma, \gamma: [0,\infty)\to A$ emanating from the same point $o\in A$, we define the comparison angle between $\sigma(t)$, $\gamma(s)$, for $t,s>0$ at $o$, denoted $\angle^0 o_{\sigma(t)}^{\gamma(s)}$, as 
 \begin{equation}
     \angle^0 o_{\sigma(t)}^{\gamma(s)}\defeq\arccos \frac{\dist_A^2(o,\sigma(t))+\dist_A^2(o,\gamma(s))-\dist_A^2(\sigma(t),\gamma(s))}{2\dist_A(o,\sigma(t))\dist_A(o,\gamma(s))}.
 \end{equation}
$ (t,s)\mapsto \angle^0 o_{\sigma(t)}^{\gamma(s)}$ is monotone nonincresing w.r.t. both variables $t,s$ when fix the other. It follows that $t\mapsto\angle^0 o_{\sigma(t)}^{\gamma(t)}$ is also monotone nonincresing. The angle between $\sigma$ and $\gamma$, denoted $\angle(\sigma,\gamma)$, is defined as $\lim_{t,s\to 0} \angle^0 o_{\sigma(t)}^{\gamma(s)}$ and the angle between $\sigma$ and $\gamma$ at infinity, denoted by $\angle_{\infty}(\sigma,\gamma)$, is defined as $\lim_{t\to \infty} \angle^0 o_{\sigma(t)}^{\gamma(t)}$. Both $\angle$ and $\angle_{\infty}$ are distance functions on the set of all rays emanating from $o$. The (equivalence classes of) rays emanating from $o$ equipped with distance $\angle_{\infty}$ is an Alexandrov space of curvature lower bound $1$, which is called the ideal boundary of $(A,\dist_A)$ at $o$, when equipped with distance $\angle$ it is the space of directions at $o$ which is also an Alexandrov space of curvature lower bound $1$.

\subsection{Asymptotic cones and splitting for large scales}\label{subsec:AsymCone}
For a pointed metric space $(X,\dist, x)$, an asymptotic cone at $x$, whenever exists, is defined as a pGH limit of $(X,r_i^{-1}\dist, x)$ for some sequence of scales $r_i\to\infty$. Recall that if $(X,\dist,\meas)$ is an $\RCD(K,N)$ space, then for any $r>0$, $(X,r^{-1}\dist,\meas)$ is an $\RCD(r^2K,N)$ space. With the stability of $\RCD(K,N)$ condition it immeditely follows that by passing to a subsequence if necessary, asymptotic cones always exist at every point of an $\RCD(0,N)$ space and they are also $\RCD(0,N)$ spaces. An asymptotic cone of a $\ncRCD(0,N)$ space is noncollapsed if and only if the asymptotic volume ratio is not zero, due to Theorem \ref{thm:ncRCD} and a simple scaling. Recall \eqref{eq:AVR0}.

For Alexandrov spaces of nonnegative curvature, the asymptotic cone at any point is unique and it is a metric cone over its ideal boundary. See \cite{antonelli2023isoperimetric}*{Theorem 2.11} and references therein.

However, for $\RCD(0,N)$ spaces, as pointed out in the last paragraph of section \ref{sec:intro}, very little structure theory of their asymptotic cones is known. However, it is known that every nonconstant linear growth harmonic function induces a splitting factor $\R$ in the asymptotic cone. We state the original version for manifolds \cite{CheegerColdingMinicozzi}, nevertheless an extension to $\RCD$ setting is essentially done in \cite{honda2021sobolev}*{Theorem 4.8}.
 
\begin{theorem}\label{thm:AsymSplitting}
    Let $(M,g)$ be an Riemannian open $n$-manifold. If the dimension of the space of linear growth harmonic functions is $k+1$, $k\in \N^+$, then any asymptotic cone splits $\R^k$. 
\end{theorem}

To close this subsection we recall the following splitting theorem for large scales and varying points by Kapovitch-Wilking \cite{KapovitchWilking}*{Lemma 2.1}, which serves as a basis of the rescaling theorem. We slightly extend it to the $\RCD$ setting.

\begin{lemma}\label{lem:splitLarge}
Let $(X_i,\dist_i,\meas_i,x_i)$ be a sequence of $\RCD(-\eps_i,N)$ spaces, where $\eps_i\to 0^+$. Given $k\in \N^+$, and $r_i\to \infty$, if there exists $(b^i_1,\ldots, b^i_k): B_{r_i}(x_i)\to \R^k$ so that the following is satisfied
\begin{itemize}
    \item For each $i\in \N$ and $j=1,2\ldots, k$ the function $b^i_j$ is in $D(\Delta, B_{r_i}(x_i))$, the domain of the local Laplacian on $ B_{r_i}(x_i)$. 
    \item For any fixed $R>0$,
    \begin{equation}\label{eq:intzero}
         \dashint_{B_R(x_i)} \sum_{j,l=1}^k|\nabla b^i_j\cdot\nabla b^i_l-\delta_{j,l}|+R^2\sum_{j=1}^k(\Delta b^i_j)^2\mr{d}\meas_i\to 0, \quad\text{as $i\to \infty.$}
    \end{equation}
\end{itemize}
Then $(X_i, \dist_i,(\meas_i(B_1(x_i)))^{-1}\meas_i,x_i)$ pmGH subconverges to an $\RCD(0,N)$ space $(\R^k\times Y,(0,y))$ with product distance and some limit measure.
\end{lemma}

We refer the readers to \cite{AHlocal} for relevant definitions.

\begin{proof}
    It follows from \eqref{eq:intzero} that for each $j$, $\|\nabla b^i_j\|_{L^2(B_R)}$ and $\|\Delta b^i_j\|_{L^2(B_R)}$ are uniformly bounded, so $b^i_j$ converges in $H^{1,2}_{\mr{loc}}$ \cite{AHlocal}*{Theorem 4.4} as $i\to\infty$ to a $b_j\in D(\Delta,Y)$. It again follows from \eqref{eq:intzero} that $b_j$ is harmonic and $|\nabla b_j|=1$, $\nabla b_j\cdot \nabla b_l=0$\ a.e. in the limit measure for $j\neq l$. The Sobolev-to-Lipschitz property shows that $b_j$ can also be chosen to be $1$-Lipschitz. The conclusion follows from the splittiing theorem \cite{Gigli_splitting}.
\end{proof}


\subsection{Aspherical manifolds and asymptotic dimension}
 A manifold is asperical if all of its higher homotopy groups $\pi_k$, $k\ge2$, vanish. For dimension $2$, $\mb S^2$, $\R P^2$ are only closed surfaces that are not aspherical. It is conjectured by Gromov-Lawson that closed aspherical manifolds do not admit any metric of positive scalar curvature. Some partial progresses are known \cites{Chodosh_Li_Soapbubble, gromov2020NoPcs5D}. With an additional assumption on the fundamental group, the following is proved.

\begin{theorem}[\cite{hypersphericityDra}*{p.157 Corollary}]\label{thm:asdimpsc}
    Let $M$ be a closed aspherical manifold. If the fundamental group $\pi_1(M)$ as a metric space equipped with the word metric has finite asymptotic dimension. Then $M$ does not admit a metric of positive scalar curvature.
\end{theorem}

We introduce the notion of asympototic dimension following closely the survey \cite{asdimsurvey}, see also the references therein for the motivations and history of this definition. 

\begin{definition}[\cite{asdimsurvey}*{p.1270 Definition}]\label{def:asdim}
    Let $X$ be a metric space. For $n\in N$, we say the asymptotic dimension of $X$ does not exceed $n$, if for every uniformly bounded open cover $\mc V$, there exist a uniformly bounded open cover $\mc U$ with multiplicity at most $n+1$ so that $\mc V$ is a refinement of $\mc U$. We define the asymptotic dimension of $X$ as $\mr{asdim}(X)\defeq\min\{n\in \N:\text{asymptotic dimension of $X$ does not exceed $n$}\}$.
\end{definition}

Here the multiplicity of an open cover $\mc U$ is the maximum integer $k$ so that every $x\in X$ is contained in the intersection of at most $k$ open sets in $\mc U$. We say an open cover $\mc U$ is uniformly bounded if $\sup_{U\in \mc U}\mr{diam}(U)<\infty$. When it comes to the asymptotic dimension of a group, then it is viewed as a metric space endowed with the word metric.


 We collect the following properties of the asymptotic dimension that will be used in the proof of Theorem \ref{thm:betti}.

\begin{proposition}\label{prop:asdim}
    The following are true.
\begin{enumerate}
    \rm{\item\label{item:Z}(\cite{asdimsurvey}*{p.1271 Example})} $\mr{asdim}(\mb Z)=1$;
    \rm{\item\label{item:product}(\cite{asdimsurvey}*{Theorem 32})} If $X,Y$ are metric spaces. Then $\mr{asdim}(X\times Y)\le \mr{asdim}(X)+\mr{asdim}(Y)$;
    \rm{\item\label{item:FiniteExtension}(\cite{asdimsurvey}*{Corollary 54 (1)})} If $\Gamma$ is a finitely generated group and $\Gamma'\le \Gamma$ is a subgroup of finite index, then $\mr{asdim}(\Gamma')=\mr{asdim}(\Gamma)$;
    \rm{\item\label{item:sum}(\cite{asdimsurvey}*{Theorem 63})} If there is an exact sequence of groups
    \begin{equation*}
        1\rightarrow  F\rightarrow G \rightarrow H\rightarrow 1,
    \end{equation*}
    where $G$ is finitely generated, then $\mr{asdim}(G)\le \mr{asdim}(F)+\mr{asdim}(H)$
    \rm{\item\label{item:hyperbolic}(\cite{asdimsurvey}*{Theorem 87}}) Every finitely generated hyperbolic group has finite asymptotic dimension. 
    
\end{enumerate}
\end{proposition}

\section{Fibration for large first Betti number}\label{sec:Betti}
We discuss some topological properties of Ricci limit (or $\RCD$) spaces. The aim is to prove Theorem \ref{thm:fibration}. It is motivated by the following question studied and disproved (for $b_1<\mr{dimension}$) by Anderson \cite{AndersonFibration}.
\begin{question}
    Given a closed Riemannian $n$-manifold $(M,g)$, is there an $\eps>0$ so that if $\Ric_g\ge -\eps$ and $\mr{diam}(M,\dist_g)\le 1$, then $M$ is a fiber bundle over $\mb T^{b_1(M)}$?
\end{question}

This fibration-over-torus problem was also studied earlier in \cites{YamaguchiAlmostRic,YamaguchiCollapsingPinching,BingWangBetti}, under various curvature bounds and extra assumptions. Anderson's counterexample \cite{AndersonFibration} shows that even when $|\Ric_g|\le \eps$, a fibration over $\mb T^{b_1(M)}$ for any $b_1(M)<n$ may not exist if the volume is collapsing, nevertheless Theorem \ref{thm:fibration} confirms that if we add the volume noncollpased assumption and a Ricci upper bound, for large $b_1$, we still retain the fibration over $\mb T^{b_1}$.

We recall that the universal cover of an $\RCD$ space, hence a Ricci limit space, is defined (by satisfying the universal property of covering spaces) and studied in \cite{MondinoWeiUni}. It is now known that the definition of the universal cover for $\RCD$ spaces coincides with the one for manifolds, proved by Wang.

\begin{theorem}[\cites{Wang_RicciSemilocal,wangRCDsemilocal}]\label{thm:Deck=fund}
Let $(X,\dist,\meas)$ be an $\RCD(K,N)$ space. Then the deck transform group of the universal cover $(\tilde X,\tilde\dist,\tilde \meas)$ is isomorphic to $\pi_1(X)$. Moreover, let $(X_i,\dist_i,\meas_i)$ be a sequence of compact $\RCD(K,N)$ spaces with $\sup_i\mr{diam}_{\dist_i}(X_i)<D<\infty$. If $(X_i,\dist_i,\meas_i)$ mGH converges to $(X,\dist,\meas)$, then there is a surjective homomorphism $\pi_1(X_i)\to \pi_1(X)$.
\end{theorem}

\begin{remark}\label{rem:lowerBetti}
     A homomorphism takes commutators to commutators, so the sujective homomorphism $\pi_1(X_i)\to \pi_1(X)$ induces also a sujective homomorphism $H_1(X_i)\to H_1(X)$, meaning the first Betti number is lower semicontinuous under mGH convergence and uniform upper diameter bound.
\end{remark}

 On the other hand, \cite{RodriguezFundamental}*{Theorem 5} shows that for noncollapsed spaces converging to a noncollapsed space, the lower bound on the first Betti number does not decrease.
 
 \begin{theorem}[\cite{RodriguezFundamental}*{Theorem 5}]\label{thm:upperBetti}
     Let $(X_i,\dist_i,\meas_i)$ be a sequence of compact $\RCD(K,N)$ spaces with $\mr{essdim}(X_i)=n\le N$, $\sup_i\mr{diam}_{\dist_i}(X_i)<D<\infty$ and $b_1(X_i)\ge r$. If $(X_i,\dist_i,\meas_i)$ mGH converges to $(X,\dist,\meas)$ with $\mr{essdim}(X)=m\le n$, then $b_1(X)\ge r+ m-n$.
 \end{theorem}
 
 Combining Remark \ref{rem:lowerBetti} and Theorem \ref{thm:upperBetti} results in the following. 
 
 \begin{proposition}\label{prop:stableBetti}
     Let $(X_i,\dist_i,\haus^N)$ be $\ncRCD(K,N)$ with $\sup_i\mr{diam}_{\dist_i}(X_i)<D<\infty$ and $\inf_i \haus^N(X_i)>v>0$. If $(X_i,\dist_i,\haus^N)$ GH converge to $(X,\dist,\haus^N)$, then $ \liminf_{i\to \infty}b_1(X_i)=b_1(X)$.
 \end{proposition}

 \begin{proof}
     Let $r=\liminf_i b_1(X_i)$. There exists a subseqence of $(X_i,\dist_i,\haus^N)$ (which converges to the same limit) we do not relabel, so that $b_1(X_i)=r$ for large $i$, because a convergent sequence of positive finite integers must stablize. It follows from Remark \ref{rem:lowerBetti} that $ b_1(X)\le r$ and Theorem \ref{thm:upperBetti} that $ b_1(X)\ge r$. We have completed the proof.
 \end{proof}
 

\begin{theorem}\label{thm:fibration}
    Let $(M,g)$ be a closed Riemannian $n$-manifold with $\mr{diam}_g(M)<D<\infty$, $\vol_g(M)>v>0$ and $\Ric_g\le K<\infty$ for some $K\ge 0$. If the first Betti number $b_1(M)=n-j$, $j=1,2,3$, then there exists an $\eps\defeq \eps(D,v,K)>0$ so that when $\Ric_g\ge -\eps$, $M$ is a topological fiber bundle over $\mb T^{b_1}$.
\end{theorem}

\begin{proof}
    We argue by contradiction. Assume there exists $\eps_i\to 0^+$ and $(M_i,g_i)$ with $\mr{diam}_{g_i}(M)<D<\infty$, $\vol_{g_1}(M)>v>0$ and $-\eps_i\le\Ric_{g_i}\le K<\infty$, so that $b_1(M_i)=n-j$ but $M$ is not a fiber bundle over $\mb T^{n-j}$. Then we get a GH limit space $(X,\dist,\haus^n)$ of $(M_i,g_i)$ which is $\ncRCD(0,n)$. It follows from Proposition \ref{prop:stableBetti} that $b_1(X)=n-j$. By \cite{MondinoWeiUni}*{Theorem 1.3}, the number of $\R$ factors splits in universal cover of $X$ is at least the first Betti number of $X$ given the identification between the group of deck transform (also called the revised fundamental group in \cite{MondinoWeiUni}) and the fundamental group $\pi_1(X)$, Theorem \ref{thm:Deck=fund}. 
    
    Along with this convergence, there is also an associated equivarient Gromov-Hausdorf (eqGH in short) convergence of the Riemannian universal cover $(\tilde M_i,\tilde g_i,\pi_1(M_i))$ where $\pi_1(M_i)$ acts isometrically. More precisely, fix $p_i\in M_i$ and a lifting $\tilde p_i\in \tilde M_i$, there exists $(Y, \dist_Y, y, G)$, $G\le \mr{Isom}(Y)$, so that $(\tilde M_i,\tilde g_i, \tilde p_i,\pi_1(M_i))$ eqGH converges to $(Y, \dist_Y, y, G)$ with $X= Y/G$ isometrically as an orbit space. Meanwhile, let $H$ be the group generated by all isotropy subgroups of $G$, then $Y/H$ is the universal cover of $X$, by \cite{PanWang_universal}*{Theorem 1.4, Corollary 6.3}. The identity component subgroup $G_0$ in there is trivial for noncollapsed spaces since $G$ ,thus $H$, is a discrete group. To proceed, we consider $3$ cases.
    \begin{itemize}
        \item $j=1$. The universal cover of $X$ is the Euclidean $\R^n$,  then it follows that $Y=\R^n$ and $G=\pi_1(X)$, which in turn implies $X=\R^n/ \pi_1(X)$, where $\pi_1(X)\le \mr{Isom}(\R^n)$ acts freely and proper discontinuously, so $X$ is a flat manifold, hence smooth. we can then appeal to \cite{BingWangBetti}*{Theorem 2.1} to conclude that $X$ is a fiber bundle over $\mb T^{n-1}$. The intrinsic Riefenberg theorem \cite{Cheeger-Colding97I}*{Appendix A} yields that for large $i$, $M_i$ and $X$ are homeomorphic, providing also a fiber bundle structure for $M_i$, a contradiction.  
        \item $j=2$. The universal cover of $X$ splits off an $\R^{n-2}$-factor and it is simply connected, so it is either $\R^n$ or $\R^{n-2}\times \mb S^2$. In the former case the proof is exactly the same as in the previous item. In the latter case, for some metric $\dist_{\mb S^2}$ the triple $(\mb S^2,\dist_{\mb S^2},\haus^2)$ carries a $\ncRCD(0,2)$ structure. Notice that by \cite{CN15_codim4}, $\mc{S}(X)=\mc{S}_{n-4}(X)$, and this property is purely local hence is lifted to the universal cover, which means $\mc R(\mb S^2)=\mb S^2$. Similarly, we claim that $Y=\mb S^2\times \R^{n-2}$ or $H$ is trivial. Indeed, notice that taking quotient by nontrivial isotropy group strictly decreases the diameter of the space of direction hence decrease the volume density. If there is a nontrivial isotropy subgroup at some $y\in Y$, then in the quotient $Y/H$ the projection of $y$ must be singular by \eqref{eq:density}, but there is no singular point in $\mb S^2\times \R^{n-2}$, a contradiction. It then follows that $X=\mb S^2\times \R^{n-2}/\pi_1(X)$ also has no singular point since the action is free and properly discontinuous, so $X$ is a $C^{1,\alpha}$ manifold. We will show that it is a fiber bundle over $\mb T^{n-2}$. Consider the covering space corresponding to $H_1(X)$, which is $\bar X\defeq (\R^{n-2}\times \mb S^2)/[\pi_1(X),\pi_1(X)]$, it follows from the induction step of the proof of \cite{MMP_Torus}*{Section 4} (especially (24) therein) that $\bar X$ splits off $\R^{n-2}$, then $\bar X=\R^{n-2}\times N^2$ for some compact $\ncRCD(0,2)$ space $N^2$. The Abelian group $H_1(X)=\mb Z^{n-2}\oplus T$ acts isometrically on $\bar X$. Here, $T$ denotes the torsion part of $H$ which is a finite sum of finite Abelian groups. Note that $\mr{Isom}(\R^{n-2}\times N^2)=\mr{Isom}(\R^{n-2})\times \mr{Isom} (N^2)$, let $P_1$ be the projection $\mr{Isom}(\R^{n-2})\times \mr{Isom} (N^2)\to\mr{Isom}(\R^{n-2})$, and $P_2$ be the projection $\mr{Isom}(\R^{n-2})\times \mr{Isom} (N^2)\to\mr{Isom}(N^2)$. We see that $P_1(H_1(X))$ acts by translation, i.e., $P_1(T)$ acts trivially on $\R^{n-2}$. Indeed, first the translation generated $\mb Z^{n-2}$ spans $\R^{n-2}$, and the action of $P_1(T)$ commutes with all translation as they are Abelian subgroups of $\mr{Isom}(\R^{n-2})$. Then write an arbitrary $P_1(T)$ action as a $(n-2)\times (n-2)$ matrix $A$, and let the linearly independent translations be $v_1,\ldots, v_{n-2}$. The commutativity reads $A(x+v_i)=Ax+v_i$, i.e.\ , $A v_i=v_i$, for $i=1,\ldots ,n-2$. So $A$ is similar to the identity and has finite order, which shows $A$ is the identity. It then follows that $X=\bar X/H_1(X)=(\mb T^{n-2}\times N^2/\mb P_2(Z^{n-2}))/T$. It is a fiber bundle over $\mb T^{n-2}$. 
        Again by the intrinsic Reifenberg theorem, $X$ is homeomorphic to $M_i$ for large $i$, a contradiction. 
        
        \item $j=3$.  The only new case is when the universal cover of $X$ is $\R^{n-3}\times N^3$, where $N^3$ is compact simply connected $C^{1,\alpha}$ manifold for some $\alpha\in (0,1)$. It also carries a $\ncRCD(0,3)$ structure. Again we only need to show that the isotropy is trivial. Suppose there is a point with nontrivial isotropy subgroup, we can blow up at the point and see that the tangent space is $\R^3$ quotient a nontrivial isometric action, then it is not isometric to $\R^3$, which contradicts $\mc{S}(X)=\mc{S}_{n-4}(X)$. We consider again the covering space corresponding to $H_1(X)$, denoted by $\bar X$. It splits off $\R^{n-3}$ and the polynomial growth order of $H_1(X)$ is exactly $n-3$, so $X$ splits as $\R^{n-3}\times Y^3$ for some compact $\RCD(0,3)$ space $Y$ which is also a $C^{1,\alpha}$ manifold. The same proof applies to show that $H_1(X)$ acts on the $\R^{n-3}$ factor by translation. The fiber bundle structure comes from the fact that the quotient map is a locally trivial fibration and the base is $\mb T^{n-3}$ since $H_1(X)$ acts on the $\R^{n-3}$ factor by translation.
        \end{itemize}
\end{proof}

\begin{remark}
	It is worth pointing out that in the previous theorem, for the case $b_1(M)=n-1$, a Ricci curvature upper bound is not needed, as flatness automatically implies smoothness. Furthermore, if $b_1(M)=n-1$ then from the proof we see that a finite cover of $M$ is $\mb T^n$, hence $M$ is diffeomorphic to a infranilmanifold. This is a different version of \cite{YamaguchiAlmostRic}*{Theorem 1}, where we replaced the upper sectional curvature bound by the lower volume bound.
\end{remark}

\section{Line splitting at infinity 
}\label{sec:line}
In this section, we discuss some elementary yet less explored aspects of nonnegative sectional curvature. The goal is to prove Theorem \ref{thm:splitting}.

 The new ingredient for nonnegative sectional curvature is a line splitting theorem that links the asymptotic cone at a point and the limit space at infinity. This could be interesting in its own right.

\begin{theorem}\label{thm:splitting}
     Let $(M,g)$ be an open $n$-manifold of nonnegative sectional curvature. If the asymptotic cone at $x\in A$ has dimension $m\in \N\cap [1,n]$, then there exists a sequence of points $p_i\to \infty$, so that the pointed Gromov-Hausdorff (pGH in short) limit of $(M,g,p_i)$ splits $m$-lines.
\end{theorem}
The inverse of the above theorem is not true, as illustrated by the following example.
\begin{example}
    The number of lines split from the limit space at infinity can be arbitrarily lager than the dimension of the asymptotic cone, as can be seen from the elliptic paraboloid $\{(x_1,x_2,\ldots,x_{n+1})\in\R^{n+1}: x_{n+1}=\sum_{i=1}^n x_i^2\}$, where $n\ge 2$. It has the half line as its asymptotic cone at $0$ but for some sequence of points diverging to infinity, the limit space is $\R^n$. 
\end{example}

We first generalize a well-known fact about splitting for a manifold of nonnegative sectional curvature. It is originally stated as follows.

\begin{proposition}\label{prop:asympsplitting}
 Let $(M,g)$ be an open $n$-manifold of nonnegative sectional curvature. If the asymptotic cone of $M$ splits a line then $M$ also splits a line. 
\end{proposition}

Let us view it from a slightly different perspective. Since the asymptotic cone is a metric cone, the fact that it splits a line is equivalent to the fact that the cone tip is in the interior of a geodesic. With this point of view, we find that the cone tip can be replaced by any other points in the asymptotic cone and the splitting then holds at infinity. 

\begin{lemma}\label{lem:splittingatinfinity}
Let $(M,g)$ be an open $n$-manifold of nonnegative sectional curvature, and $(C,\dist_C,o)$ be the asymptotic cone of $M$ at some point with cone tip $o$ and $p\in C$ be any point other than $o$. For any sequence $(M,r_i^{-2}g, p_i)$ that pGH converges to $(C,\dist_C, p)$, where $\{r_i\}_{i\in \N}$ is a sequence of positive numbers so that $r_i\to \infty$, every possible pGH limit space of a subsequence of $(M,g,p_i)$ splits a line. 
\end{lemma}

The proof is essentially the same as that of Proposition \ref{prop:asympsplitting}, which can be found for examples in recent \cite{andoni2022ancient}*{Lemma 2.3} or \cite{Isoregion}*{Theorem 4.6}. The proof here is in the spirit of \cite{andoni2022ancient}*{Lemma 2.3}

\begin{proof}
Since $(C,\dist_C,o)$ is a metric cone and $p$ is not $o$, there is a ray $\sigma:[0,\infty)\to C$ emanating from $o$ passing through $p$. We can assume that $p=\sigma(2)$. There exists a sequence of points $q^-_i$ converging to $\sigma(1)\defeq q^-$ and a sequence $p^+_i$ converging to $\sigma(3)\defeq q^+$. Clearly, we have 
\begin{equation}\label{eq:asympdistance}
    r_i^{-1}\dist_g(p_i, q^-_i)\to 1,\quad r_i^{-1}\dist_g(p_i, q^+_i)\to 1, \quad r_i^{-1}\dist_g(q^-_i, q^+_i)\to 2,\quad \text{as $i\to \infty$}.
\end{equation}
Let $\sigma^-_i$ (resp. $,\sigma^+_i$) be a geodesic joining $p_i$ and $q^-_i$(resp. $q^+_i$), which has length $O(r_i)$. Observe that by the monotonicity of comparison angles, it follows from a direct computation based on \eqref{eq:asympdistance} that
\begin{equation}\label{eq:angleineq}
    \angle(\sigma^-_i, \sigma^+_i)\ge \arccos{\frac{\dist_g^2(p_i,q^-_i)+\dist_g^2(p_i,q^+_i)-\dist_g^2(q^-_i,q^+_i)}{2\dist_g(p_i,q^-_i)\dist_g(p_i,q^+_i)}}\to \pi \quad \text{as $i\to\infty$}.
\end{equation}
By Gromov's precompactness theorem $(M,g,p_i)$ pGH (sub)converges to an Alexandrov space $(A,\dist_A, p_\infty)$. It can be seen from the non-branching property that $\sigma^-_i$ (resp. $\sigma^+_i$) converges pointwise to a ray emanating from $p_\infty$ denoted by $\sigma^-$ (resp. $\sigma^+$). We claim that the concatenation of $\sigma^-$ and $\sigma^+$ is a (minimizing) geodesic. To show this, it suffices to show that  $\angle_{\infty}(\sigma^-,\sigma^+)=\pi$. 

To this end, first fix a $j\in \N$. We choose points $\tilde q^-_j$ on $\sigma^-$ so that $\dist_A(p_\infty, \tilde q^-_j)=\dist_g(p_j,q^-_j)$, they are uniquely determined by their distances to $p_\infty$. For each sufficiently large $i>j$ depending on $j$, $\tilde q^-_j$ can be pulled back via the GH approximation to a point on $\sigma^-_i$ since we have assumed that $\sigma^-_i$ converges pointwise to $\sigma^-$. We denote this point by $q^-_{j,i}$. The same procedure can be done when replacing $-$ with $+$, we denote the resulting point by $q^+_{j,i}$. Note that $\dist_g(p_i,q^-_i)\ge \dist_g (p_i, q^-_{j,i})$, along with \eqref{eq:angleineq}, the monotonicity of comparison angles applied to angles formed by $q^-_{j,i},p_i,q^+_{j,i}$ and $q^-_{i},p_i,q^+_{i}$ along $\sigma^-_i$ and $\sigma^+_i$ yields that 
\begin{align*}
    \frac{\dist_g^2(p_i,q^-_{j,i})+\dist_g^2(p_i,q^+_{j,i})-\dist_g^2(q^-_{j,i},q^+_{j,i})}{2\dist_g(p_i,q^-_{j,i})\dist_g(p_i,q^+_{j,i}) }&\le \frac{\dist_g^2(p_i,q^-_i)+\dist_g^2(p_i,q^+_i)-\dist_g^2(q^-_i,q^+_i)}{2\dist_g(p_i,q^-_i)\dist_g(p_i,q^+_i)} \notag\\ 
    &= \cos(\pi-\eps_i)\le-1+\frac12 \eps_i^2,
\end{align*}
for some $\eps_i \to 0^+$ as $i\to\infty$. This in turn implies that
\begin{align*}
    \dist_g^2(q^-_{j,i},q^+_{j,i})\ge (\dist_g(q^-_{j,i},p_i)+\dist_g(p_i,q^+_{j,i}))^2- \eps_i^2 \dist_g(q^-_{j,i},p_i)\dist_g(q^+_{j,i},p_i).
\end{align*}

Note that both $\dist_g(q^-_{j,i},p_i),\dist_g(q^+_{j,i},p_i)$ are bounded independent of $i$. Letting $i\to \infty$, together with triangle inequality, it holds that $\dist_A(\tilde q^-_{j},\tilde q^+_{j})= \dist_A(\tilde q^-_{j},p_\infty)+\dist_A(p_\infty,\tilde q^+_{j})$. Since this holds for every $j$, we see that $$
\angle_{\infty}(\sigma^-,\sigma^+)=\lim_{j\to\infty}{\arccos{\frac{\dist_A^2(p_\infty,\tilde q^-_j)+\dist_A^2(p_\infty,\tilde q^+_j)-\dist_A^2(\tilde q^-_j,\tilde q^+_j)}{2\dist_A(p_\infty, \tilde q^-_j)\dist_A(p_\infty,\tilde q^+_j)}}}=\pi,
$$ 
 as desired. Now that there is a line in $A$, the splitting  theorem then asserts that $A$ splits a line.
\end{proof} 

 As an application we give an alternative proof of the fact that any pGH converging sequence $(M,g,p_i)$ with $p_i$ diverging to infinity splits a line in its limit space, obtained in \cite{antonelli2023isoperimetric}*{Lemma 2.29}. Our method also reveals the relation between this splitting and the fact that any asymptotic cone of $M$ is a metric cone hence for every point other than the tip there is a ray passing through it.

\begin{corollary}\label{cor:splittingatinfinity}
Let $(M,g)$ be an open $n$-manifold of nonnegative sectional curvature and $p_i\to\infty$. If $(M,g,p_i)$ pGH converges to an Alexandrov space $(A,\dist_A, x)$, then $A$ splits a line.
\end{corollary}

\begin{proof}
    Take $y\in M$, and let $\dist_g(y, p_i)=2r_i$. By assumption $r_i\to \infty$. $(M,r_i^{-2}g,y)$ pGH converges to a metric cone $(C,\dist_C, o)$ and $p_i$ converges to a point $p_\infty$ with $\dist_C(o,p_\infty)=2$. Let $\sigma$ be the ray emanating from $o$ passing through $p_\infty$. We can find a geodesic $\sigma_i$ contained in the ball $(B_{3/2}(p_i), r_i^{-2}g)$ converging with the rescaling to $\sigma|_{(1/2,7/2)}$, then the proof follows exactly the same as Lemma \ref{lem:splittingatinfinity}. We conclude that $\sigma_i$ converges without rescaling to a line and $x$, the pGH limit of $p_i$, is on this line.
\end{proof}

\begin{remark}
In the proof of corollary \ref{cor:splittingatinfinity}, the geodesic joining the cone tip $o$ and $p_\infty$ is unique thanks to the structure of metric cones, so we can easily find a sequence of geodesics converging to it. However, the uniqueness of geodesic or the metric cone structure is not crucial. In fact, the non-branching property of Alexandrov spaces implies that any geodesic in a smoothable Alexandrov space is a limit geodesic, i.e. it can be realized as a pointwise limit of a sequence of geodesics in the approximating sequence of manifolds.
\end{remark}

We are in position to prove Theorem \ref{thm:splitting}.

\begin{proof}[Proof of Theorem \ref{thm:splitting}] 
Let $(C,\dist_C, o)$ be the asymptotic cone at $x$. It is an Alexandrov space of Hausdorff dimension $m\ge 1$, in particular, it has a regular point $p$ other than the cone tip $o$. To simplify the notation we assume that $\dist_C(o,p)=2$. Then there exists a sequence of real numbers $r_i\to \infty$ so that $(B_s(p_i),r_i^{-2}g, p_i)$ pGH converges to $(B_s(p),\dist_C, p)$ for any $s\in[0,1]$. By a diagonal argument, there exists a sequence of real numbers $t_i\to \infty$ as $i\to\infty$ so that $(B_1(p_i),t_i^{-2}g,p_i)$ pGH converges to $(B_1(0), |\cdot|, 0)\subset \R^n$ which is the tangent cone at $p$. Now that $0$ is in the interior of $m$ perpendicular geodesic segments, by repeatedly applying Lemma \ref{lem:splittingatinfinity}, we get $m$ lines $\sigma_j$, $j=1,2,\ldots,m$, in the pGH limit space $(A,\dist_A, x)$ of the sequence $(M,g,p_i)$, where $\sigma_j:\R\to A$ corresponds to the line segment on $i$-th coordinate passing through $0$ in the rescaled limit $B_1(0)$. Let $\sigma_j^- \text{ (resp.$\sigma_j^+$)} :[0,\infty]\to A$ be such that $\sigma_j^-(t)=\sigma(-t)$ (resp. $\sigma_j^+(t)=\sigma(t)$). It follows from the monotonicity that 

$$
\angle(\sigma^-_j, \sigma^+_k)\ge \angle_\infty(\sigma^-_j, \sigma^+_k)=\pi/2,
$$
 and also 
$$
 \pi-\angle(\sigma^-_j, \sigma^+_k)=\angle(\sigma^+_j ,\sigma^+_k)\ge\angle_\infty(\sigma^+_j, \sigma^+_k)\ge\pi/2,
$$
  hence $\angle(\sigma^-_j, \sigma^+_k)=\pi/2$ for any $j,k\in\{1,2,\ldots, m\}$, $j\neq k$. Which means every line $\sigma_j$ is perpendicular to each other, so each of them splits a different $\R$ factor, we have complete the proof.
\end{proof}

\begin{remark}\label{rem:ric}
    We proved that $2$ perpendicular geodesics segments in an asymptotic cone induce $2$ different splitting 
 $\R$ factors for manifolds of nonnegative sectional curvature. This heavily relies on the monotonicity of comparison angles. We do not know if Theorem \ref{thm:splitting} holds for nonnegative Ricci curvature, except for $m=1,n$. The following example may help to illustrate one of the difficulties we face.

 It is observed by Pan--Wei \cite{pan2022examplesbusemann}*{section 1.2} that a surface of revolution $(\R\times \mb S^1,g\defeq\mr{d}r^2+h^2(r)\mr{d}\theta^2)$ can be isometric to a totally geodesic and geodesically complete submanifold in a manifold of positive Ricci curvature, where $h$ can be chosen as $\frac{1}{\ln(2+r^2)}$. Consider a sequence of points $p_k=(k, \theta_0)$, for some fixed $\theta_0\in\mb S^1$ and $k\in \N$. Take a pair of antipodal points $\{ x,-x\}$ in $\mb S^1$, let $\sigma_k^+$ (resp. $\sigma_k^-$) be the geodesic joining $p_k$ and $(0,x)$ (resp. $(0,-x)$). The pGH limit $(\R\times \mb S^1,g,p_k)$ is a line, so $\sigma_k^+$ and $\sigma_k^-$ merge to the same ray in the limit space at infinity, even if the distance between the end points $\dist_g((0,x),(0,-x))=\frac1{\ln 2}>0$. 
\end{remark}




\section{The two-dimension vanishing theorems}\label{sec:collapse}
This section is devoted to the proof of Theorem \ref{thm:codim2collapseRic} and of Theorem \ref{thm:betti}. 
Since in this section we deal with many scalings, we will denote a geodesic ball of radius $r>0$ w.r.t. Riemmanian metric $g$ (resp. distance $\dist$) as $B_r^g$ (resp. $B_r^{\dist}$). 

First, let us prove Theorem \ref{thm:codim2collapseRic}. We first recall the rescaling theorem of Kapovitch-Wilking \cite{KapovitchWilking}*{Theorem 5.1}. Here we simplify the setting and only take the first part of it for our purpose. This part can be proved verbatim in $\RCD$ setting because of Lemma \ref{lem:splitLarge} and the propagation of splitting, see for example \cite{deltasplitingRCD}*{Proposition 1.6}. We do not pursue the complete proof. 
\begin{theorem}[\cite{KapovitchWilking}*{Theorem 5.1}]\label{thm:rescaling}
    Let $(X_i,\dist_i,\meas_i, x_i)$ be a sequence of pointed $\RCD(0,N)$ spaces. 
    Assume that $(X_i,\dist_i,\meas_i, x_i)$ pmGH converges to $(\R^k,|\cdot|,\leb^k)$ for $0\le k<\inf_i\mr{essdim}(X_i)$, $k\in \N^+$, then there exists a compact $\RCD(0,N-k)$ space $(D,\dist_D, \meas_D,d)\neq\{\mr{pt}\}$, a sequence of scales $\lambda_i\to \infty$, and a sequence of Borel sets $G_1(x_i)$ with $\meas_i(G_1(x_i))\ge (1-i^{-1})\meas_i(B_1(x_i))$, so that for any $y_i\in G_1(x_i)$, $(X_i,\lambda_i\dist_i,({\meas(B_1^{\lambda_i\dist_i}(y_i))})^{-1}\meas_i, y_i)$ pmGH subconverges to an $\RCD(0,N)$ of isometric type $\R^{k}\times D$. 
\end{theorem}

\begin{proof}[Proof of Theorem \ref{thm:codim2collapseRic}]
    We argue by contradiction. First notice that the essential dimension cannot be $n$, if it was, then by Theorem \ref{thm:essdimgap} this space is noncollpased. However, recall that we have \eqref{eq:AVR0}, i.e., the asymptotic volume ratio is zero, so it must be collapsed by Theorem \ref{thm:ncRCD}, a contradiction. We assume that there exists a point $q\in M$ and a sequence of scales $r_i\to \infty$, so that $(M,g_i\defeq r_i^{-2}g, (\vol_{g_i}(B^{g_i}_1(q)))^{-1}\vol_{g_i},q)$ pmGH converges to a 
asymptotic cone $(X,\dist, \meas,x)$ of essential dimension $n-1$. Take a regular point $y\in \mc R_{n-1}(X)$, then there exists a sequence of points $\{p_i\}_{i\in\N}$ in $M$ diverging to infinity so that $B_s^{g_i}(p_i)$ pGH converges to $B^{\dist}_s(y)$ for any $s\in [0,1]$. By a diagonal argument there exists a sequence of scales $s_i\to \infty$ so that $(M,s_i^{-2} g, p_i)$ pGH converges to $(\R^{n-1},|\cdot|,0)$, the tangent cone at $y$. Applying Theorem \ref{thm:rescaling} for $(M,s_i^{-2} g, p_i)$, we find another sequence of scales $\lambda_i\to \infty$, and a sequence of points $\{q_i\}_{i\in \N}$ in $M$ so that $(M,\bar g_i\defeq ({\lambda_i}{s_i}^{-1})^{2} g, q_i)$ pGH subconverges (without relabeling) to a $\RCD(0,n)$ space $\R^{n-1}\times D$ with product distance and some limit measure. The splitting theorem \cite{Gigli_splitting} implies that $D$ is a $\RCD(0,1)$ space. Combine with the fact $D\neq \{\mr{pt}\}$, the classification theorem \cite{KL15} in turn implies that $D$ has essential dimension $1$, then $\R^{n-1}\times D$ is noncollapsed, in particular the measure supporting the $\RCD(0,n)$ structure can be taken as the Hausdroff measure $\haus^n$ induced by the product distance. By taking again a subsequence we can assume $\lim_{i\to\infty }\frac{\lambda_i}{s_i}$ exists in $[0,\infty]$. We claim that 
\begin{equation}\label{eq:limitfinite}
    \lim_{i\to\infty }\frac{\lambda_i}{s_i}<\infty.
\end{equation}

 If $\lim_{i\to \infty}\frac{\lambda_i}{s_i}=\infty$. Notice that the limit space is noncollapsed, we have from the first item of Theorem \ref{thm:ncRCD} and Bishop-Gromov inequality that
\begin{equation}
    0<\inf_i \vol_{\bar g_i}g(B^{\bar g_i}_1(q_i))\le \sup_i \vol_{\bar g_i}(B^{\bar g_i}_1(q_i))\le \omega_n,
\end{equation}
where $\omega_n$ is the volume of the unit ball in $\R^n$. Denote by $(0,d)\in \R^{n-1}\times D$ the limit of $q_i$, the volume convergence for noncollapsed spaces then implies that 
\begin{equation}
   \lim_{i\to\infty} \vol_{\bar g_i}(B^{\bar g_i}_R(q_i))=\haus^n(B_R((0,d)))\quad \forall R>0.
\end{equation}
However, it follows from Bishop-Gromov inequality that for $R>1$,
\begin{equation}
    \lim_{i\to\infty} \vol_{\bar g_i}(B_R^{\bar g_i}(q_i))=\lim_{i\to\infty} \frac{\vol_{g}(B_{s_i\lambda_i^{-1}R}^{ g}(q_i))}{(s_i\lambda_i^{-1})^n}\ge \inf_i B^{ g}_1(q_i)R^n= O(R^n),
\end{equation}
while $\haus^n(B_R((0,d)))$ has order $O(R^{n-1})$ as $R\to \infty$, a contradiction.

 Now that we have \eqref{eq:limitfinite}, it follows that $\Sc_{\bar g_i}\ge ({s_i}{\lambda_i}^{-1})^2K$, which remains uniformly positive. Moreover, the limit space is noncollapsed, so $(M,\bar g_i, q_i)$ is a sequence that satisfies the assumptions of Theorem \ref{thm:n-2splitting}, then the limit space of this sequence cannot split $\R^{n-1}$, a contradiction.   
\end{proof}

Now we turn to the first Betti number.

\begin{proof}[Proof of Theorem \ref{thm:betti}]
First we prove the upper bound. Assume the contrary, then there exists a sequence $(M_i,g_i)$ and $\eps_i\to 0^+$ so that $(M_i, g_i)$ satisfies that $\Ric_{g_i}\ge-\eps_i$, $\Sc_g\ge K>0$, $\mr{diam}(M_i,\dist_{g_i})< D$, and $\vol_{g_i}(M_i)>v$ but $b_1(M_i)\ge n-1$. By Theorem \ref{thm:upperBetti}, $(M_i, g_i)$ pGH subcongverges to a compact $\ncRCD(0,n)$ space with $b_1(X)\ge n-1$. Then the proof proceeds exactly the same as that of Theorem \ref{thm:fibration}. 
Essentially we get that the Riemannian universal cover of $M_i$ GH subconverges to a noncollapsed Ricci limit space that is $\R^{n}$, contradicting Theorem \ref{thm:n-2splitting}. 

If in addition $b_1(M)=n-2$ and the Ricci curvature is bounded from above, then $M$ is homeomorphic to a fiber bundle over $\mb T^{b_1(M)}=\mb T^{n-2}$ by Theorem \ref{thm:fibration}. 

We show that $F$ cannot be aspherical hence must be $\mb S^2$ or $\R P^2$. Recall that $\pi_1(M)$ is finitely generated. 

If $F$ was aspherical, then from the exact sequence of homotopy groups, $M$ is also aspherical. We show that $\pi_1(M)$ has finite asymptotic dimension, so $M$ does not admit any metric of positive curvature thanks to Theorem \ref{thm:asdimpsc}, a contradiction. Indeed, one has that
\begin{equation}
    \pi_2(\mb T^{n-2})\cong 1\rightarrow \pi_1(F)\rightarrow \pi_1(M)\rightarrow \pi_1(\mb T^{n-2})\cong \mb Z^{n-2}\rightarrow 1.
\end{equation}
It follows from Proposition \ref{prop:asdim} \eqref{item:sum} that $\mr{asdim}(\pi_1(M))\le \mr{asdim}(\mb Z^{n-2})+\mr{asdim}(\pi_1(F))$. Being a closed $2$-surface, there are only $2$ cases for $F$, either $F$ is hyperbolic, then $\pi_1(F)$ is a hyperbolic group, or $F$ is flat. In the former case $\mr{asdim}(\pi_1(F))$ is finite due to Proposition \ref{prop:asdim} \eqref{item:hyperbolic}. In the latter case, 
in fact, Bieberbach theorem says for any compact flat $n$-manifold $F$, there is always a finite indexed subgroup $\mb Z^n=\pi_1(F)\cap (\{\mr{Id}\}\times \R^n)\le \pi_1(F)\cap \mr{Isom}(\R^n)$. Now that $\mr{asdim} (\mb Z)=1$ (Proposition \ref{prop:asdim} \eqref{item:Z}), recall Proposition \ref{prop:asdim} \eqref{item:product} and \eqref{item:FiniteExtension}, it follows that $\mr{asdim} (\pi_1(F))<\infty$ in either case. 

Now, we assume in addition that that $3\le n\le 7$, $M$ is orientable and $F=\mb S^2$. We proceed to find an embedded $\mb S^2$ with the desired area upper bound. Consider the projection map $f:M\to \mb T^{n-2}$, Gromov \cite{gromov2020NoPcs5D}*{Theorem 1} claims that there exists a closed $2$-dimensional submanifold $Y\subset X$ that is homologous to the fiber $\mb S^2=f^{-1}(t)$, for all $t\in \mb T^{n-2}$. In dimension $2$ this in turn implies $Y$ is diffeomorphic to $\mb S^2$. Furthermore, one can produce a torical symmetrization of $Y$ by a descent method. We briefly recall the argument. We start by finding a stable minimal surface $\Sigma_1$ in the homology class of $f^{-1}(\mb T^{n-3})$ for some $T^{n-3}\hookrightarrow T^{n-2}$. The restriction $f:\Sigma_1\hookrightarrow \mb T^{n-2}\to \mb T^{n-3}$ is a continuous map. Furthermore, we can equip $\Sigma_1\times \mb S^1$ with a warped product metric $g_{\Sigma_1}+\phi_1^2\mr{d}u_1^2$ for some $\phi_1$ so that its scalar curvature lower bound is the same as $M$. Inductively implementing this procedure, we get a manifold $(X,g_X)=(Y\times \mb T^{n-2}, g_Y+\sum_{i=1}^{n-2}\phi_i^2\mr{d}u_i^2)$ so that $\Sc_{g_X}\ge\Sc_{g}\ge K>0$, where $\phi_i$'s are warping functions coming from the first eigenfunctions of the Jacobi operators and $u_i$'s are coordinates on $\mb S^1$. Meanwhile we have a chain of embedding maps $S^2\cong Y \defeq\Sigma_{n-2}\hookrightarrow \Sigma_{n-1}\hookrightarrow\cdots\hookrightarrow \Sigma_1\hookrightarrow M $. 
Once we have $(X,g_X)=(Y\times \mb T^{n-2}, g_Y+\sum_{i=1}^{n-2}\phi_i^2\mr{d}u_i^2)$ with $\Sc_{g_X}\ge K>0$, the area upper bound of $Y$ follows from \cite{ZhuJT_area}*{Lemma 2.3} and the area upper bound of $Y$ is an upper bound for the minimal embedded sphere in $M$. Moreover, in the case of equality, once we in addition have the chain of embedding maps, the rigidity follows from the proof of \cite{ZhuJT_area}*{Theorem 1.2} in Chapter 3. We have complete the proof.

\end{proof}



Finally, we prove Theorem \ref{thm:dim_n-1}. The ideas of the proof are due to Jiayin Pan. We record here the new ingredients from Pan--Ye \cite{panyesplitting}.

\begin{proposition}[\cite{panyesplitting}*{Corollary 3.2}]\label{prop:blowdown}
    Let $(M,g)$ be an open $n$-manifold with $\Ric_g\ge 0$ and $p\in M$. Suppose $\mr{Isom}(X)$ contains a closed subgroup $\Gamma\defeq \mb Z^b$ for some $b\in \N^+$. For any $r_i\to\infty$, consider the equivariant convergence
        \[
        (r_i^{-1}M,p,\Gamma)\to (Y,y,G).
        \]
        Then the limit group $G$ contains a closed subgroup $\R^b$. In particular the orbit $Gy$ has topological dimension at least $b$.
    
\end{proposition}

\begin{proposition}[\cite{panyesplitting}*{corollary 3.3}]\label{prop:blowup}
    Let $(X,\dist)$ be a Ricci limit space and $p\in X$. Suppose $\mr{Isom}(X)$ contains a closed subgroup $G\defeq \R^b$ for some $b\in \N^+$. For any $r_i\to\infty$, consider the equivariant convergence
        \[
        (r_i X,p,G)\to (Y,y,H).
        \]
        Then the limit group $H$ contains a closed subgroup $\R^b$. In particular the orbit $Hy$ has topological dimension at least $b$. 
\end{proposition}

\begin{proof}[Proof of Theorem \ref{thm:dim_n-1}]
    Let $\Gamma=\pi_1(M)/[\pi_1(M),\pi_1(M)]$, $b\defeq b_1(M)$, $h\defeq h_1(M)$, and $\tilde M$ be the universal cover of $M$. Consider the covering space correspond to $\Gamma$, $\overline M\defeq \tilde M/[\pi_1(M),\pi_1(M)]$ and $\bar p\in \overline{ M}$ the lift of $p$. Note that $\mr{Isom}(\overline{M})$ contains a closed subgroup $\mb Z^b\le \Gamma$. Using Theorem \ref{thm:AsymSplitting}, we have the following equivariant GH convergence of blow-down. 
     \[
     \begin{tikzcd}
(r_i^{-1}\overline M, \bar p, \Gamma) \arrow[r, "\mr{GH}"] \arrow[d,"\pi"]
& (\R^{h-1}\times \overline Y, (0,\bar y) ,G) \arrow[d, "\pi"] \\
(r_i^{-1} M, p)\arrow[r, "\mr{GH}"]
& |[]| (\R^{h-1}\times Y,(0,y)),
\end{tikzcd}
     \]
  By Proposition \ref{prop:blowdown}, $G$ contains a closed subgroup $\R^b$ and clearly $G$ acts only on $\overline Y$. Meanwhile by Theorem \ref{thm:codim2collapseRic}, $\R^{h-1}\times \overline Y$ has essential dimension at most $n-2$ since $\overline{M}$ satisfies all the assumptions there when $M$ does. Now find a regular point $q$ in $\R^{h-1}\times \overline Y$, take the tangent of $(\R^{h-1}\times \overline Y,q, G)$, we get $(\R^{h-1}\times\R^k,0,H)$, where $k\in \N^+$. By Proposition \ref{prop:blowup}, $H$ has a closed subgroup $\R^b$ and the orbit $H\cdot 0$ sits inside $\R^k$ and homeomorphic to $\R^n$, which implies $k\ge b$. Then by the lower semiconinuity of the essential dimension, $k+h-1\le n-2$, which gives $h+b\le n-1$ as desired. The proof of $b_1(M)\le n-3$ goes exactly the same way. Assume on the contrary $b_1(M)\ge n-2$. Using the same notations we have the following equivariant GH convergence.
 \[
     \begin{tikzcd}
(r_i^{-1}\overline M, \bar p, \Gamma) \arrow[r, "\mr{GH}"] \arrow[d,"\pi"]
& ( \overline Y, \bar y ,G) \arrow[d, "\pi"] \\
(r_i^{-1} M, p)\arrow[r, "\mr{GH}"]
& |[]| (Y,y),
\end{tikzcd}
\]
and $G$ has a closed $\R^{n-2}$ subgroup thanks to Proposition \ref{prop:blowdown}. Take a regular point $q\in Y$ and its lift $\bar q\in \overline Y$, we get an equivariant GH convergence by taking the tangent cone at $q$ and $\bar q$ for some sequence scaling $s_i\to\infty$ as follows 
 \[
     \begin{tikzcd}
(s_i\overline Y, \bar q, G) \arrow[r, "\mr{GH}"] \arrow[d,"\pi"]
& ( \R^k\times Z, (0,z) ,H) \arrow[d, "\pi"] \\
(s_i Y, q)\arrow[r, "\mr{GH}"]
& |[]| (\R^k,0),
\end{tikzcd}
\]
where $k$ is the essential dimension of $Y$. Since $Y$ is not a point we infer that $k\ge 1$. Thanks to Proposition \ref{prop:blowup} $H$ also contains a closed $\R^{n-2}$ subgroup, and the orbit $H\cdot z$ sits inside $Z$. By taking a regular point in $\R^k\times Z$ and taking tangent cone again, we see that the tangent cone has essential dimension $k+n-2\ge n-1$, this also contradicts Theorem \ref{thm:codim2collapseRic} by the lower semicontinuity of essential dimension because $\overline Y$ has essential dimension $n-2$.

There is also an alternative proof of $b_1(M)\le n-3$. We use Anderson's argument in \cite{andersontopology}. Let $(\tilde M,\tilde g)$ be the Riemannian universal cover of $(M,g)$. To show $b_1\le n-3$ it suffices to show that the polynomial growth rate of any finitely generated subgroup of $\pi_1(M)$ is at most $n-3$. More precisely, fix a base point $p\in M$ and let $\tilde p$ be its lift. Let $\Gamma$ be a finitely generated subgroup of $\pi_1(M,p)$ with a symmetric generating set $S$, and associated word metric $\dist_S$. For any $R>0$, let $\Gamma(R)\defeq\{g\in \Gamma| \dist_S(g,e)\le R\}$. It suffices to show $|\Gamma(R)|\le C R^{n-3}$ for some constant $C$ and large $R$. By the argument in the proof of \cite{andersontopology}*{Theorem 1.1}, there exists a constant $c\defeq c(n,M)>0$ so that 
\begin{equation}\label{eq:poly}
    |\Gamma(R)|\le \frac{\vol_{\tilde g}(B_{cR}(\tilde p))}{\vol_g(B_R(p))}.
\end{equation}
Recall the proof that manifolds with nonnegative Ricci curvature having at least linear growth, it shows $\vol_g(B_r(p))\ge c(n)\vol_g(B_1(p)) r$. Our assumption $\vol_g(B_1(p))\ge v>0$ implies that the lower bound of 
$\vol_g(B_r(p))$ does not depend on $p$. On the other hand, it is shown in \cite{WZZZ_PSC_RLS}*{Theorem 1.6} there exists universal constant $c(n)$ so that for any $R\ge 1$
\[
\inf_{p\in M} \vol_g(B_R(p))\le cR^{n-2}.
\]
Since $|\Gamma(R)|$ also does not depend on the base point $p\in M$, by taking infimum in \ref{eq:poly} we get
\[
|\Gamma(R)|\le C(n,v)R^{n-3}.
\]
This is enough to conclude $b_1(M)\le n-3$ thanks to Anderson \cite{andersontopology}*{Theorem 1.3}.
\end{proof}

Theorem \ref{thm:codim2collapse} follows directly from Theorem \ref{thm:codim2collapseRic}, we can use the more elementary splitting procedure, Theorem \ref{thm:splitting}.




\begin{remark}
 In the setting of Theorem \ref{thm:dim_n-1}, the rigidity when $b_1(M)=n-3$ is an interesting question. However it is not clear in this case if the universal cover will split off $\R^{n-3}$. If we strengthen our assumption to that $M$ has nonnegative sectional curvature, then the universal cover $(\tilde M,\tilde g)$ of $(M,g)$ satisfies exactly one of the following.
\begin{itemize}
\item $(\tilde M,\tilde g)$ is isometric to $\R^{n-2}\times \mb S^2$, where $\R^{n-2}$ is Euclidean and $\mb S^2$ carries a metric of sectional curvature lower bound $K/2$.
\item $(\tilde M,\tilde g)$ is isometric to	$\R^{n-3}\times \R^3$, where $\R^{n-3}$ is Euclidean and $\R^3$ has only one end, carrying a metric of scalar curvature lower bound $K$.
\end{itemize}
We provide a sketch of proof. Notice that the first Betti number carries over to the soul of $M$, hence $b_1(M)= n-3$ implies that $(\tilde M,\tilde g)$ splits an Euclidean $\R^{n-3}$ factor isometrically. Write $\tilde M=\R^{n-3}\times N^3$. It follows that $N^3$ is noncompact. Otherwise $\text{Isom}(\tilde M)=\text{Isom}((\R^{n-3}))\times \text{Isom}(N^3)$, and by the argument in the proof of Theorem \ref{thm:fibration}, we get that $M=(\mb T^{n-3}\times N^3)/(\pi_1(M)/\mb Z^{n-3})$, which implies that $M$ is compact, a contradiction. Next, from the classification of open $3$-manifolds with nonnegative Ricci curvature \cite{Liu_3Dclassification}, $N^3$ either splits as $\R\times \mb S^2$ with the scalar curvature lower bound being the sectional curvature lower bound of $\mb S^2$, or $N^3$ is diffeomorphic to $\R^3$ with only one end. This completes the proof of the above classification. 

For $N^3=\R^3$ with metric $g_N$ of $\text{Sc}_{g_N}>K$, and having only one end, we can see the geometrical consequences of positive scalar curvature by the pGH limit $(\R^3, g_N, p_i)$ for $p_i\to \infty$, which is a kind a limit model of its end. Let $x_0\in N^3$ be its soul, then $\dist(\cdot,x_0)$ has no critical points except for $x_0$ itself. It follows that every geodesic ball centered at $x_0$ is homeomorphic (hence diffeomorphic) to the Euclidean ball, so the end of $N$ is diffeomorphic to $[0,\infty)\times \mb S^2$. 

From now on we assume that the pGH limit of $(N^3,g_N, p_i)$ exists. First notice that the limit space must split a line because of Corollary \ref{cor:splittingatinfinity}. If $(N^3,g_N, p_i)$ is noncollapsing, then Theorem \ref{thm:n-2splitting} and \cite{ZhuZhu23}*{Propsition 3.1} (in particular the orientability excludes $\R P^2$) implies that the limit space is $\R\times \mb S^2$ with $\mb S^2$ carrying a metric of curvature lower bound $K/2$ in the sense of Alexandrov, which give rise to a diameter bound of this $\mb S^2$. This can be viewed as a diameter upper bound for the $\mb S^2$ slices in the end by the almost splitting theorem. If $(\R^3,\tilde g, p_i)$ is collapsing, the most important observation was already made in \cite{Chodosh_Chao_volumegrowth}*{Lemma 2.2, claim 1}, that is, if the points $p_i$ are taken along a ray, writing the limit space as $\R\times Y$, then $Y$ has bounded diameter. We observe that from the proof of Lemma \ref{lem:splittingatinfinity}, for arbitrary sequence of $p_i\to \infty$ there exists unit speed minimizing geodesic $\gamma_i:[0, \frac12 \dist(p_i,x_0)]\to N$ so that $(N^3,g_N, p_i)$ and $(N^3,g_N, \gamma_i(\frac14 \dist(p_i,x_0)))$ converge to the same limit space, then for large $i$, the argument of \cite{Chodosh_Chao_volumegrowth}*{Lemma 2.2 claim 1} can still be applied to $\gamma_i$. Thus, the limit of $(N^3,g_N, p_i)$ can only possibly be $\R\times \mb S^1$, $\R\times [0,a]$, $a\in(0,\infty)$ or $\R\times \{\text{pt}\}$, i.e., $\R^2$ and $\R\times [0,\infty)$ are ruled out. We do not expect that $\R\times \mb S^1$ can actually be a limit space. For example this may be seen by the local fibration theorem \cite{ShioyaYamaguchi_collapsed3Dmanifolds}*{Theorem 3.5} for $3$-dimensional collapsed manifolds. however, this is not related to scalar curvature so we do not pursue this direction. It is easy to see that every other candidate can be realized. 
\end{remark}

\textbf{Disclosure.} There is no conflict of interests in the present paper. There is no data associated to the present paper.

\bibliographystyle{amsalpha}
\bibliography{lineandpsc}

\begin{bibdiv}
\begin{biblist}

\bib{AHlocal}{article}{
      author={Ambrosio, Luigi},
      author={Honda, Shouhei},
       title={Local spectral convergence in {${\rm RCD}^*(K,N)$} spaces},
        date={2018},
        ISSN={0362-546X},
     journal={Nonlinear Anal.},
      volume={177},
      number={part A},
       pages={1\ndash 23},
         url={https://doi.org/10.1016/j.na.2017.04.003},
      review={\MR{3865185}},
}

\bib{andersontopology}{article}{
      author={Anderson, Michael~T.},
       title={On the topology of complete manifolds of nonnegative {R}icci
  curvature},
        date={1990},
        ISSN={0040-9383},
     journal={Topology},
      volume={29},
      number={1},
       pages={41\ndash 55},
         url={https://doi.org/10.1016/0040-9383(90)90024-E},
      review={\MR{1046624}},
}

\bib{AndersonFibration}{article}{
      author={Anderson, Michael~T.},
       title={Hausdorff perturbations of {R}icci-flat manifolds and the
  splitting theorem},
        date={1992},
        ISSN={0012-7094},
     journal={Duke Math. J.},
      volume={68},
      number={1},
       pages={67\ndash 82},
         url={https://doi.org/10.1215/S0012-7094-92-06803-7},
      review={\MR{1185818}},
}

\bib{Isoregion}{article}{
      author={Antonelli, Gioacchino},
      author={Bru\`e, Elia},
      author={Fogagnolo, Mattia},
      author={Pozzetta, Marco},
       title={On the existence of isoperimetric regions in manifolds with
  nonnegative {R}icci curvature and {E}uclidean volume growth},
        date={2022},
        ISSN={0944-2669},
     journal={Calc. Var. Partial Differential Equations},
      volume={61},
      number={2},
       pages={Paper No. 77, 40},
         url={https://doi.org/10.1007/s00526-022-02193-9},
      review={\MR{4393128}},
}

\bib{antonelli2023isoperimetric}{article}{
      author={Antonelli, Gioacchino},
      author={Pozzetta, Marco},
       title={Isoperimetric problem and structure at infinity on alexandrov
  spaces with nonnegative curvature},
        date={2023},
      eprint={2302.10091},
}

\bib{asdimsurvey}{article}{
      author={Bell, G.},
      author={Dranishnikov, A.},
       title={Asymptotic dimension},
        date={2008},
        ISSN={0166-8641},
     journal={Topology Appl.},
      volume={155},
      number={12},
       pages={1265\ndash 1296},
         url={https://doi.org/10.1016/j.topol.2008.02.011},
      review={\MR{2423966}},
}

\bib{BrenaGigliHondaZhu}{article}{
      author={Brena, Camillo},
      author={Gigli, Nicola},
      author={Honda, Shouhei},
      author={Zhu, Xingyu},
       title={Weakly non-collapsed $\mathrm{RCD}$ spaces are strongly
  non-collapsed},
        date={2023},
     journal={Journal für die reine und angewandte Mathematik (Crelles
  Journal)},
      volume={2023},
      number={794},
       pages={215\ndash 252},
         url={https://doi.org/10.1515/crelle-2022-0071},
}

\bib{deltasplitingRCD}{article}{
      author={Bru\`e, Elia},
      author={Pasqualetto, Enrico},
      author={Semola, Daniele},
       title={Rectifiability of {$\mathrm {RCD}(K,N)$} spaces via
  {$\delta$}-splitting maps},
        date={2021},
        ISSN={2737-0690},
     journal={Ann. Fenn. Math.},
      volume={46},
      number={1},
       pages={465\ndash 482},
         url={https://doi.org/10.5186/aasfm.2021.4627},
      review={\MR{4277822}},
}

\bib{BrueSemola20Constancy}{article}{
      author={Brué, Elia},
      author={Semola, Daniele},
       title={{Constancy of the Dimension for ${\rm RCD}(K, N)$ Spaces via
  Regularity of Lagrangian Flows}},
        date={2020},
        ISSN={0010-3640 1097-0312},
     journal={Communications on Pure and Applied Mathematics},
      volume={73},
      number={6},
       pages={1141\ndash 1204},
}

\bib{BBI01}{book}{
      author={Burago, D.},
      author={Burago, Y.},
      author={Ivanov, S.},
       title={{A Course in Metric Geometry}},
        type={Book},
   publisher={American Mathematical Society},
        date={2001},
        ISBN={9780821821299},
         url={https://books.google.ca/books?id=afnlx8sHmQIC},
}

\bib{Cai_Largemanifold}{article}{
      author={Cai, Mingliang},
       title={{On {G}romov's large {R}iemannian manifolds}},
        date={1994},
        ISSN={0046-5755},
     journal={Geom. Dedicata},
      volume={50},
      number={1},
       pages={37\ndash 45},
         url={https://doi.org/10.1007/BF01263649},
      review={\MR{1280793}},
}

\bib{Cheeger-Colding97I}{article}{
      author={Cheeger, Jeff},
      author={Colding, Tobias~H.},
       title={On the structure of spaces with {R}icci curvature bounded below.
  {I}},
        date={1997},
        ISSN={0022-040X},
     journal={J. Differential Geom.},
      volume={46},
      number={3},
       pages={406\ndash 480},
         url={http://projecteuclid.org/euclid.jdg/1214459974},
      review={\MR{1484888}},
}

\bib{CheegerColdingMinicozzi}{article}{
      author={Cheeger, Jeff},
      author={Colding, Tobias~H.},
      author={Minicozzi, William.~P., II},
       title={Linear growth harmonic functions on complete manifolds with
  nonnegative {R}icci curvature},
        date={1995},
        ISSN={1016-443X},
     journal={Geom. Funct. Anal.},
      volume={5},
      number={6},
       pages={948\ndash 954},
         url={https://doi.org/10.1007/BF01902216},
      review={\MR{1361516}},
}

\bib{CN15_codim4}{article}{
      author={Cheeger, Jeff},
      author={Naber, Aaron},
       title={{Regularity of {E}instein manifolds and the codimension 4
  conjecture}},
        date={2015},
        ISSN={0003-486X},
     journal={Ann. of Math. (2)},
      volume={182},
      number={3},
       pages={1093\ndash 1165},
         url={https://doi.org/10.4007/annals.2015.182.3.5},
      review={\MR{3418535}},
}

\bib{Chodosh_Li_Soapbubble}{article}{
      author={Chodosh, Otis},
      author={Li, Chao},
       title={Generalized soap bubbles and the topology of manifolds with
  positive scalar curvature},
        date={2024},
        ISSN={0003-486X,1939-8980},
     journal={Ann. of Math. (2)},
      volume={199},
      number={2},
       pages={707\ndash 740},
         url={https://doi.org/10.4007/annals.2024.199.2.3},
      review={\MR{4713021}},
}

\bib{Chodosh_Chao_volumegrowth}{article}{
      author={Chodosh, Otis},
      author={Li, Chao},
      author={Stryker, Douglas},
       title={Volume growth of 3-manifolds with scalar curvature lower bounds},
        date={2023},
        ISSN={0002-9939,1088-6826},
     journal={Proc. Amer. Math. Soc.},
      volume={151},
      number={10},
       pages={4501\ndash 4511},
         url={https://doi.org/10.1090/proc/16521},
      review={\MR{4643334}},
}

\bib{CN11}{article}{
      author={Colding, Tobias~H.},
      author={Naber, Aaron},
       title={Characterization of tangent cones of noncollapsed limits with
  lower {R}icci bounds and applications},
        date={2013},
        ISSN={1016-443X},
     journal={Geom. Funct. Anal.},
      volume={23},
      number={1},
       pages={134\ndash 148},
         url={https://doi.org/10.1007/s00039-012-0202-7},
      review={\MR{3037899}},
}

\bib{CN12}{article}{
      author={Colding, Tobias~Holck},
      author={Naber, Aaron},
       title={Sharp {H}\"{o}lder continuity of tangent cones for spaces with a
  lower {R}icci curvature bound and applications},
        date={2012},
        ISSN={0003-486X},
     journal={Ann. of Math. (2)},
      volume={176},
      number={2},
       pages={1173\ndash 1229},
         url={https://doi.org/10.4007/annals.2012.176.2.10},
      review={\MR{2950772}},
}

\bib{DPG17}{article}{
      author={De~Philippis, Guido},
      author={Gigli, Nicola},
       title={Non-collapsed spaces with {R}icci curvature bounded from below},
        date={2018},
        ISSN={2429-7100},
     journal={J. \'{E}c. polytech. Math.},
      volume={5},
       pages={613\ndash 650},
         url={https://doi.org/10.5802/jep.80},
      review={\MR{3852263}},
}

\bib{hypersphericityDra}{article}{
      author={Dranishnikov, A.~N.},
       title={On hypersphericity of manifolds with finite asymptotic
  dimension},
        date={2003},
        ISSN={0002-9947},
     journal={Trans. Amer. Math. Soc.},
      volume={355},
      number={1},
       pages={155\ndash 167},
         url={https://doi.org/10.1090/S0002-9947-02-03115-X},
      review={\MR{1928082}},
}

\bib{Gigli_splitting}{article}{
      author={Gigli, Nicola},
       title={An overview of the proof of the splitting theorem in spaces with
  non-negative {R}icci curvature},
        date={2014},
     journal={Anal. Geom. Metr. Spaces},
      volume={2},
      number={1},
       pages={169\ndash 213},
         url={https://doi.org/10.2478/agms-2014-0006},
      review={\MR{3210895}},
}

\bib{GMS13}{article}{
      author={Gigli, Nicola},
      author={Mondino, Andrea},
      author={Savar\'{e}, Giuseppe},
       title={Convergence of pointed non-compact metric measure spaces and
  stability of {R}icci curvature bounds and heat flows},
        date={2015},
        ISSN={0024-6115},
     journal={Proc. Lond. Math. Soc. (3)},
      volume={111},
      number={5},
       pages={1071\ndash 1129},
         url={https://doi.org/10.1112/plms/pdv047},
      review={\MR{3477230}},
}

\bib{Gromov_metric_inequality}{article}{
      author={Gromov, Misha},
       title={{Metric inequalities with scalar curvature}},
        date={2018},
        ISSN={1016-443X},
     journal={Geom. Funct. Anal.},
      volume={28},
      number={3},
       pages={645\ndash 726},
         url={https://doi.org/10.1007/s00039-018-0453-z},
      review={\MR{3816521}},
}

\bib{Gromov_four_lectures}{article}{
      author={Gromov, Misha},
       title={Four lectures on scalar curvature},
        date={2019},
      eprint={1908.10612},
         url={https://arxiv.org/abs/1908.10612},
}

\bib{gromov2020NoPcs5D}{article}{
      author={Gromov, Misha},
       title={No metrics with positive scalar curvatures on aspherical
  5-manifolds},
        date={2020},
      eprint={2009.05332},
}

\bib{honda2021sobolev}{article}{
      author={Honda, Shouhei},
      author={Sire, Yannick},
       title={Sobolev mappings between {RCD} spaces and applications to
  harmonic maps: a heat kernel approach},
        date={2023},
        ISSN={1050-6926,1559-002X},
     journal={J. Geom. Anal.},
      volume={33},
      number={9},
       pages={Paper No. 272, 87},
         url={https://doi.org/10.1007/s12220-023-01334-6},
      review={\MR{4603300}},
}

\bib{hua2016harmonic}{article}{
      author={Hua, Bobo},
      author={Kell, Martin},
      author={Xia, Chao},
       title={Harmonic functions on metric measure spaces},
        date={2016},
      eprint={1308.3607},
}

\bib{BingWangBetti}{article}{
      author={Huang, Shaosai},
      author={Wang, Bing},
       title={Rigidity of the first betti number via ricci flow smoothing},
        date={2021},
      eprint={2004.09762},
}

\bib{KapovitchWilking}{article}{
      author={Kapovitch, Vitali},
      author={Wilking, Burkhard},
       title={Structure of fundamental groups of manifolds with ricci curvature
  bounded below,},
        date={2011},
      eprint={1105.5955},
}

\bib{kitabeppu2017sufficient}{article}{
      author={Kitabeppu, Yu},
       title={A sufficient condition to a regular set being of positive measure
  on {$\rm{RCD}$} spaces},
        date={2019},
        ISSN={0926-2601},
     journal={Potential Anal.},
      volume={51},
      number={2},
       pages={179\ndash 196},
         url={https://doi.org/10.1007/s11118-018-9708-4},
      review={\MR{3983504}},
}

\bib{KL15}{article}{
      author={Kitabeppu, Yu},
      author={Lakzian, Sajjad},
       title={Characterization of low dimensional {$\mathrm{RCD}^*(K,N)$}
  spaces},
        date={2016},
     journal={Anal. Geom. Metr. Spaces},
      volume={4},
      number={1},
       pages={187\ndash 215},
         url={https://doi.org/10.1515/agms-2016-0007},
      review={\MR{3550295}},
}

\bib{LiTamLinear}{article}{
      author={Li, Peter},
      author={Tam, Luen-Fai},
       title={Linear growth harmonic functions on a complete manifold},
        date={1989},
        ISSN={0022-040X},
     journal={J. Differential Geom.},
      volume={29},
      number={2},
       pages={421\ndash 425},
         url={http://projecteuclid.org/euclid.jdg/1214442883},
      review={\MR{982183}},
}

\bib{Liu_3Dclassification}{article}{
      author={Liu, Gang},
       title={3-manifolds with nonnegative {R}icci curvature},
        date={2013},
        ISSN={0020-9910},
     journal={Invent. Math.},
      volume={193},
      number={2},
       pages={367\ndash 375},
         url={https://doi.org/10.1007/s00222-012-0428-x},
      review={\MR{3090181}},
}

\bib{andoni2022ancient}{article}{
      author={Lynch, Stephen},
      author={Abrego, Andoni~Royo},
       title={Ancient solutions of {R}icci flow with type {I} curvature
  growth},
        date={2024},
        ISSN={1050-6926,1559-002X},
     journal={J. Geom. Anal.},
      volume={34},
      number={5},
       pages={Paper No. 119, 17},
         url={https://doi.org/10.1007/s12220-024-01558-0},
      review={\MR{4715385}},
}

\bib{MenguyNonpolar}{article}{
      author={Menguy, X.},
       title={Examples of nonpolar limit spaces},
        date={2000},
        ISSN={0002-9327},
     journal={Amer. J. Math.},
      volume={122},
      number={5},
       pages={927\ndash 937},
  url={http://muse.jhu.edu/journals/american_journal_of_mathematics/v122/122.5menguy.pdf},
      review={\MR{1781925}},
}

\bib{MMP_Torus}{article}{
      author={Mondello, Ilaria},
      author={Mondino, Andrea},
      author={Perales, Raquel},
       title={An upper bound on the revised first {B}etti number and a torus
  stability result for {{RCD}} spaces},
        date={2022},
        ISSN={0010-2571},
     journal={Comment. Math. Helv.},
      volume={97},
      number={3},
       pages={555\ndash 609},
         url={https://doi.org/10.4171/cmh/540},
      review={\MR{4468994}},
}

\bib{MondinoWeiUni}{article}{
      author={Mondino, Andrea},
      author={Wei, Guofang},
       title={On the universal cover and the fundamental group of an {${\rm
  RCD}^*(K,N)$}-space},
        date={2019},
        ISSN={0075-4102},
     journal={J. Reine Angew. Math.},
      volume={753},
       pages={211\ndash 237},
         url={https://doi.org/10.1515/crelle-2016-0068},
      review={\MR{3987869}},
}

\bib{MW_geometryofpsc}{article}{
      author={Munteanu, Ovidiu},
      author={Wang, Jiaping},
       title={Geometry of three-dimensional manifolds with scalar curvature
  lower bound, preprint,
  \href{https://arxiv.org/abs/2201.05595}{arXiv:2201.05595}},
        date={2022},
         url={https://arxiv.org/abs/2201.05595},
}

\bib{PanWang_universal}{article}{
      author={Pan, Jiayin},
      author={Wang, Jikang},
       title={Some topological results of {R}icci limit spaces},
        date={2022},
        ISSN={0002-9947},
     journal={Trans. Amer. Math. Soc.},
      volume={375},
      number={12},
       pages={8445\ndash 8464},
         url={https://doi.org/10.1090/tran/8549},
      review={\MR{4504644}},
}

\bib{pan2022examplesbusemann}{article}{
      author={Pan, Jiayin},
      author={Wei, Guofang},
       title={Examples of open manifolds with positive ricci curvature and
  non-proper busemann functions,},
        date={2022},
      eprint={2203.15211},
}

\bib{PanWeiHaus}{article}{
      author={Pan, Jiayin},
      author={Wei, Guofang},
       title={Examples of {R}icci limit spaces with non-integer {H}ausdorff
  dimension},
        date={2022},
        ISSN={1016-443X},
     journal={Geom. Funct. Anal.},
      volume={32},
      number={3},
       pages={676\ndash 685},
         url={https://doi.org/10.1007/s00039-022-00598-4},
      review={\MR{4431126}},
}

\bib{panyesplitting}{article}{
      author={Pan, Jiayin},
      author={Ye, Zhu},
       title={Nonnegative ricci curvature, splitting at infinity, and first
  betti number rigidity},
        date={2024},
      eprint={2404.10145},
         url={https://arxiv.org/abs/2404.10145},
}

\bib{PetruninCD}{article}{
      author={Petrunin, Anton},
       title={Alexandrov meets {L}ott-{V}illani-{S}turm},
        date={2011},
        ISSN={1867-5778},
     journal={M\"{u}nster J. Math.},
      volume={4},
       pages={53\ndash 64},
      review={\MR{2869253}},
}

\bib{RodriguezFundamental}{article}{
      author={Santos-Rodr\'iguez, Jaime},
      author={Zamora-Barrera, Sergio},
       title={On fundamental groups of {RCD} spaces},
        date={2023},
        ISSN={0075-4102,1435-5345},
     journal={J. Reine Angew. Math.},
      volume={799},
       pages={249\ndash 286},
         url={https://doi.org/10.1515/crelle-2023-0027},
      review={\MR{4595312}},
}

\bib{Shen_Largemanifold}{article}{
      author={Shen, Zhongmin},
       title={{Complete manifolds with nonnegative {R}icci curvature and large
  volume growth}},
        date={1996},
        ISSN={0020-9910},
     journal={Invent. Math.},
      volume={125},
      number={3},
       pages={393\ndash 404},
         url={https://doi.org/10.1007/s002220050080},
      review={\MR{1400311}},
}

\bib{ShioyaYamaguchi_collapsed3Dmanifolds}{article}{
      author={Shioya, Takashi},
      author={Yamaguchi, Takao},
       title={Volume collapsed three-manifolds with a lower curvature bound},
        date={2005},
        ISSN={0025-5831,1432-1807},
     journal={Math. Ann.},
      volume={333},
      number={1},
       pages={131\ndash 155},
         url={https://doi.org/10.1007/s00208-005-0667-x},
      review={\MR{2169831}},
}

\bib{Wang_RicciSemilocal}{article}{
      author={Wang, Jikang},
       title={Ricci limit spaces are semi-locally simply connected},
        date={2021},
      eprint={2104.02460},
         url={https://arxiv.org/abs/2104.02460},
}

\bib{wangRCDsemilocal}{article}{
      author={Wang, Jikang},
       title={{${\rm RCD}^*(K,N)$} spaces are semi-locally simply connected},
        date={2024},
        ISSN={0075-4102,1435-5345},
     journal={J. Reine Angew. Math.},
      volume={806},
       pages={1\ndash 7},
         url={https://doi.org/10.1515/crelle-2023-0058},
      review={\MR{4685082}},
}

\bib{WZZZ_PSC_RLS}{article}{
      author={Wang, Jinmin},
      author={Xie, Zhizhang},
      author={Zhu, Bo},
      author={Zhu, Xingyu},
       title={Positive scalar curvature meets ricci limit spaces},
        date={2022},
      eprint={2212.10416},
         url={https://arxiv.org/abs/2212.10416},
}

\bib{YamaguchiAlmostRic}{article}{
      author={Yamaguchi, Takao},
       title={Manifolds of almost nonnegative {R}icci curvature},
        date={1988},
        ISSN={0022-040X},
     journal={J. Differential Geom.},
      volume={28},
      number={1},
       pages={157\ndash 167},
         url={http://projecteuclid.org/euclid.jdg/1214442165},
      review={\MR{950560}},
}

\bib{YamaguchiCollapsingPinching}{article}{
      author={Yamaguchi, Takao},
       title={Collapsing and pinching under a lower curvature bound},
        date={1991},
        ISSN={0003-486X},
     journal={Ann. of Math. (2)},
      volume={133},
      number={2},
       pages={317\ndash 357},
         url={https://doi.org/10.2307/2944340},
      review={\MR{1097241}},
}

\bib{Zhu_Geometryofpsc}{article}{
      author={Zhu, Bo},
       title={Geometry of positive scalar curvature on complete manifold},
        date={2022},
        ISSN={0075-4102},
     journal={J. Reine Angew. Math.},
      volume={791},
       pages={225\ndash 246},
         url={https://doi.org/10.1515/crelle-2022-0049},
      review={\MR{4489630}},
}

\bib{ZhuZhu23}{article}{
      author={Zhu, Bo},
      author={Zhu, Xingyu},
       title={Optimal diameter estimate of three-dimensional {R}icci limit
  spaces},
        date={2024},
        ISSN={0002-9939,1088-6826},
     journal={Proc. Amer. Math. Soc.},
      volume={152},
      number={2},
       pages={815\ndash 821},
         url={https://doi.org/10.1090/proc/16529},
      review={\MR{4683860}},
}

\bib{ZhuJT_area}{article}{
      author={Zhu, Jintian},
       title={{Rigidity of area-minimizing {$2$}-spheres in {$n$}-manifolds
  with positive scalar curvature}},
        date={2020},
        ISSN={0002-9939},
     journal={Proc. Amer. Math. Soc.},
      volume={148},
      number={8},
       pages={3479\ndash 3489},
         url={https://doi.org/10.1090/proc/15033},
      review={\MR{4108854}},
}

\end{biblist}
\end{bibdiv}
\end{document}